\documentclass[11pt]{amsart}
\usepackage[margin=1.5in]{geometry}

\usepackage{mathrsfs}
\usepackage{amsfonts}
\usepackage{latexsym,epsfig}
\usepackage{mathabx}
\usepackage{amsmath, amscd, amsthm,amssymb}
\usepackage[pdftex]{color}
\usepackage{comment}
\usepackage{marginnote}
\usepackage[colorlinks,citecolor=blue,linkcolor=red]{hyperref}
\usepackage[nameinlink]{cleveref}
\usepackage{enumitem}
\usepackage{enumitem}

\newcommand{\RR}[0]{\mathbb{R}}
\newcommand{\HH}[0]{\mathbb{H}}

\newcommand{\Z}{\mathbb{Z}}

\newcommand{\A}[0]{\mathcal{A}}

\newcommand{\CC}[0]{\mathcal{C}}

\newcommand{\R}{\mathbb{R}}

\newcommand{\diam}{\mathrm{diam}}

\newcommand{\wt}{\widetilde}
\newcommand{\mc}{\mathcal}
\newcommand{\mf}{\mathfrak}
\newcommand{\cone}{\mathrm{cone}}

\newcommand{\mr}{\mathring}

\newcommand{\WF}{\mathrm{WF}}
\newcommand{\hm}{\text{\tiny{HM}}}
\DeclareMathOperator{\fr}{fr}

\newcommand{\cut}{\backslash \!\! \backslash}

\usepackage{hyperref}
\usepackage{pinlabel}

\newtheorem{thm}{Theorem}

\newtheorem{theorem}{Theorem}[section]
\newtheorem{lemma}[theorem]{Lemma}
\newtheorem{proposition}[theorem]{Proposition}
\newtheorem{corollary}[theorem]{Corollary}

\newtheorem{claim}{Claim}

\theoremstyle{definition}

\newtheorem{remark}[theorem]{Remark}
\newtheorem{convention}[theorem]{Convention}

\newcommand{\FF}{\mathcal{F}}

\newcommand{\C}{\mathcal{C}}

\newcommand{\Mod}{\mathrm{Mod}}

\newcommand{\orb}{\mathcal{O}}

\newcommand\tsim{\kern-.4em\sim}

\newcommand\ssm{\smallsetminus}

\newcommand{\vol}{\mathrm{vol}}

\renewcommand{\phi}{\varphi}
\renewcommand{\epsilon}{\varepsilon}

\newcommand{\TTT}{\mathbb{T}}
\newcommand{\cU}{\mathcal{U}}

\DeclareMathOperator{\intr}{int}

\title{Pseudo-Anosov flows, hyperbolic geometry, and the curve graph}
\author{Junzhi Huang and Samuel J. Taylor}

\begin{document}

\begin{abstract}
Starting with a pseudo-Anosov flow $\varphi$ on a closed hyperbolic $3$--manifold $M$ and an embedded surface $S \subset M$ that is (almost) transverse to $\varphi$, we relate the hyperbolic geometry of $M$ (e.g. volume, circumference, short geodesics) to dynamical invariants of $\varphi$ encoded by the curve graph of $S$. 
\end{abstract}

\maketitle

\setcounter{tocdepth}{1}
\tableofcontents

\section{Introduction}
\label{sec:intro}
Let $M$ be a closed hyperbolic $3$--manifold and let $\varphi$ be a pseudo-Anosov flow on $M$. The purpose of this paper is to quantify how dynamical properties of $\varphi$ influence the geometry of $M$ through the curve graph of a closed embedded surface $S\subset M$ that is (almost) transverse to $\varphi$.

\smallskip
The prototypical example is where $S$ is a cross section of $\varphi$ (i.e. $S$ is transverse to $\varphi$ and intersects every orbit), and hence $M$ fibers over the circle with $S$ as a fiber. In this setting, the restriction of the  stable/unstable foliations of $\varphi$ to $S$ are precisely the end invariants of the cover $\wt M_S$ of $\wt M$ associated to $S$. Hence, by the resolution of Thurston's Ending Lamination Conjecture, proven by Minsky \cite{ECL1} and Brock--Canary--Minsky \cite{ELC2}, the hyperbolic structure on $\wt M_S$ is totally determined by these dynamical/geometric invariants. In fact, geometric quantities like the volume of $M$ \cite{Brock2}, the circumference of $M$ \cite{biringer2019ranks, aougab2022covers}, and the location of short geodesics \cite{minsky2001bounded, viaggi2025effective} are each coarsely determined (up to constants depending only on $\chi(S)$) by combinatorial quantities that are organized by the curve graph of $S$ and its subsurfaces.  

Despite the significance of these results, very little is known in the more general setting where the surface $S$ is transverse to $\varphi$ but not a cross section (that is, it misses some orbits). 
Here, the hyperbolic manifold $\wt M_S$ is quasi-Fuchsian (\cite{cooper1994bundles,fenley1999surfaces}) and there is no obvious connection between the stable/unstable foliations of $\varphi$ and the end invariants of $\wt M_S$.

\subsection*{Results}

Let $M$ be a closed hyperbolic $3$--manifold, let $\varphi$ be an almost pseudo-Anosov flow on $M$, and let $S$ be a connected surface embedded in $M$ that is transverse to $\varphi$.
 The intersections of $S$ with the invariant stable/unstable foliations $\mc F^{s/u}$ of $\varphi$ are a pair of singular foliations $\mc F_S^{s/u}$ of $S$.  When $S$ is not a fiber surface (or more precisely, not a cross section of $\varphi$), the foliations $\mc F_S^{s/u}$ each contain a nonempty set of closed leaves $c^{s/u}$ (see \Cref{sec:transvese_shadow}).  
The pair $c^{s/u}$ are essential multicurves on $S$,
which we call the \emph{stable/unstable multicurves} of $S$. 

Informally, our main theorems state that if $c^s$ and $c^u$ are complicated with respect to one another, as viewed in the curve graph of $S$, then the geometry of $M$ near $S$ is also complicated. To state this precisely, we first introduce a notion of complexity of the cut manifold $M\cut S$. This generalizes the notion of `core complexity' originally defined when $S$ is a compact leaf of a depth one foliation of $M$ (see \cite{field2025lower}, \cite{whitfield2024short}, as well as \Cref{sec:examples}). 
\smallskip

Recall that $M\cut S$ denotes the manifold obtained by cutting $M$ along $S$. Its boundary is a disjoint union $\partial_- (M \cut S) \sqcup \partial_+ (M \cut S)$, where the flow points into, or out of, $M\cut S$ respectively. A \emph{product annulus} in $M \cut S$ is an essential properly embedded annulus that has one boundary component on each of $\partial_-(M\cut S)$ and $\partial_+(M\cut S)$.
Define the \emph{core complexity} of $M\cut S$ to be 
\begin{align*}
\mathfrak{c}(M\cut S) = 
\begin{cases}
0, & M\cut S \text{ contains a product annulus},\\
\min{-\chi(\Sigma)}, & \text{otherwise},
\end{cases}
\end{align*}
where the min is over all properly embedded surfaces $\Sigma$ in $M\cut S$ that are transverse to (a dynamic blowup of) $\varphi$ and which meet both $\partial_+ (M \cut S)$ and $\partial_- (M\cut S)$. Dynamic blowups and almost pseudo-Anosov flows are defined in \Cref{sec:pA}.
For context, $\mathfrak{c}(M\cut S)$ contributes to an \emph{additive} error term in our theorems that relate geometry and dynamics. We note that $\mathfrak{c}(M\cut S)< \infty$ when $\varphi$ is a suspension flow (i.e. $\varphi$ admits a cross section) or more generally when $S$ is a leaf of a (almost) transverse finite depth foliation (see \Cref{sec:examples}).
A cohomological characterization of when $\mathfrak{c}(M\cut S)< \infty$ will be  subject of forthcoming work \cite{HT_depthone}.

\smallskip

Our first main theorem relates the curve graph distance between $c^u$ and $c^s$ to the volume and circumference of $M$. See \Cref{th:vol} and \Cref{th:circum}, respectively.

\begin{thm}[Volume and circumference]\label{thm:vol-circumference}
Let $M$ be a closed hyperbolic $3$-manifold with a pseudo-Anosov flow $\varphi$. Let $S$ be a closed connected surface in $M$ that is almost transverse to $\varphi$, but is not a fiber. Let $c^{s/u}$ be the stable/unstable multicurves of $S$. Assume that $\mathfrak{c}(M\cut S)<\infty$.

There is a constant $k_S = k_S(\chi(S)) \ge 1$, depending only on $\chi(S)$, and a constant $k_{M \cut S} \ge 1$, depending only on $\chi(S)$ and $\mf c(M\cut S)$, so that the following inequalities hold:
\begin{enumerate}
\item $d_{\C(S)}(c^s, c^u) \le k_S \cdot \vol(M) + k_{M\cut S}$, and 
\smallskip
\item $d_{\C(S)}(c^s, c^u) \le k_S \cdot \ell_M(\gamma) + k_{M \cut S}$, where $\gamma$ is any closed geodesic that intersects $S$ essentially and $\ell_M(\gamma)$ is its length in $M$.
\end{enumerate}
\end{thm}

We remark that when $M\cut S$ contains a product annulus, $\mf c(M\cut S) = 0$ and so the constants in \Cref{thm:vol-circumference} depend only on $\vert \chi(S) \vert$.

There is also a version of \Cref{thm:vol-circumference} that starts with a finite depth foliation rather than a flow, where the multicurves $c^{s/u}$ are replaced by minimal junctures of the foliation. See \Cref{sec:examples} and in particular \Cref{cor:finitedepth} for details.

\smallskip
The second main theorem relates subsurface distances between $c^u$ and $c^s$ to short geodesics in $M$. Informally, it says that if $c^{s/u}$ are complicated with respect to one another, from the perspective of some distant subsurface $Y \subset S$, then the geodesic length of $\partial Y$ in $M$ must be small. See \Cref{th:short_curves}.

\begin{thm}[Short curves]\label{thm:short-curves}
Let $M$ be a closed hyperbolic $3$-manifold with a pseudo-Anosov flow $\varphi$. Let $S$ be a closed surface in $M$ that is almost transverse to $\varphi$ and let $Y \subset S$ be a 
subsurface of $S$. Assume that $\mathfrak{c}(M\cut S)<\infty$. Then for any $\epsilon >0$, there is a $K = K(\epsilon, \chi(S)) \ge 0$, depending only on $\epsilon$ and $\chi(S)$, such that if
\begin{itemize}
\item $d_{\C(S)}(c^{s/u} ,\partial Y) \ge k_{M\cut S}$ and 
\item $d_{\C(Y)}(c^s,c^u) \ge K + 2 \cdot  k_{M \cut S}$,
\end{itemize}
then 
\[
\ell_M(\partial Y) \le \epsilon.
\]
\end{thm}

Here, $\ell_M(\partial Y)$ denotes the length of the geodesic representative of $\partial Y$ in $M$, and $d_{\C(Y)}(c^s,c^u)$ is the subsurface distance in the curve graph of $Y$ between $c^s$ and $c^u$. See \Cref{sec:curvegraph} for details.

\begin{remark}
Our main technical result that controls the geometry of $M$ is \Cref{prop:bounded_length}. Informally, it states that each of the multicurves $c^{s/u}$ have bounded curve graph distance from curves on $S$ that have bounded hyperbolic length in $M$. One might hope to prove the stronger statement that $c^{s/u}$ themselves have bounded length, but this is not always possible. Indeed, Whitfield \cite[Theorem A]{whitfield2024short} produces examples where the `junctures' of a depth one foliation of $M$ are arbitrarily short. (Here, the junctures are essentially boundary components of a properly embedded surface in $M \cut S$ realizing the minimum core complexity.)
Since such curves necessarily cross $c^{s/u}$, in these examples the stable/unstable multicurves can be arbitrarily long. 
\end{remark}

Finally, several examples are given in \Cref{sec:examples}, including a flexible construction of manifolds and flows that realizes the hypotheses of the main theorems. 

\subsection{Connections to the literature}
Although most of the previous work that relates hyperbolic geometry, dynamics, and the curve graph comes from the fibered manifold setting (as discussed above), there has recently been an interest in extending these connections to mapping tori of endperiodic maps of infinite-type surfaces. This is most prominent in work of Field, Kim, Kent, Leininger, and Loving \cite{field2023end, field2025lower} who establish connections to hyperbolic volume of the mapping tori, and work of Whitfield \cite{whitfield2024short}, who studies the lengths of short geodesics.

The work in this paper sheds light on a different aspect of the geometry of these mapping tori. Indeed, returning to the setting of our main theorems, in the special case where $S$ is a compact leaf of a transverse depth one foliation of $M$, each component of $M \cut S$ is the mapping torus of an especially nice endperiodic map -- the first return map to a depth one leaf of the foliation using the transverse flow. 
The works cited above study the intrinsic geometry of components of $M \cut S$, especially for the unique totally geodesic structure, when it exists. By contrast, our results here are about the geometry of the original closed manifold $M$. Informally,  \Cref{thm:vol-circumference} and \Cref{thm:short-curves} focus on the geometric features of $M$ that are forced when gluing $M \cut S$ back to obtain $M$, as measured by the curve graph complexity of the stable and unstable multicurves. For more on the implications in the depth one case, see \Cref{sec:examples}. 

\subsection{Proof strategy and outline of paper}
The proof of \Cref{thm:vol-circumference} and \Cref{thm:short-curves} is divided into three steps. In the first step, we identify multicurves on $S$ that have topological significance in $M\cut S$, and have bounded intersection number with $c^{s/u}$. When $M\cut S$ admits product annuli, these are the components of Thurston's window frame of $M\cut S$ (see \Cref{sec:geom}),
and we show that product annuli can be put into a particularly nice position with respect to the flow. This is explained in \Cref{sec:annuli}.
When $M\cut S$ does not admit a product annulus but has finite core complexity, these are the boundary components of a properly embedded transverse surface with minimal complexity. In the first case, the intersection number bounds are established in \Cref{sec:annuli}, and in the second case, this is done in \Cref{sec:int_bounds}.

In the second step, we show that these multicurves have bounded length in the hyperbolic metric on $M$. This is done in \Cref{sec:geom}.
The results in the first two steps lead to \Cref{prop:bounded_length}, which states that $c^s$ and $c^u$ have bounded intersection number with bounded length curves in $S$. Finally, we obtain geometric control of $M$ by passing to the quasi-Fuchsian cover corresponding to $S$ and using standard tools from the study of curve graphs and Kleinian surface groups, replacing $c^s$ and $c^u$ with nearby bounded length curves. This is done in \Cref{sec:endgame}.

Finally, 
examples and applications are given in \Cref{sec:examples}, including a version of \Cref{thm:vol-circumference} for finite depth foliations as well as a construction where the 
dynamical quantity $d_{\C(S)}(c^s,c^u)$ can be precisely controlled.

\subsection*{Acknowledgements}
The authors thank Michael Landry and Yair Minsky for several helpful and encouraging conversations. The authors also thank Ellis Buckminster and Brandis Whitfield for helpful comments on an earlier draft.

The completion of this paper was supported by the National Science Foundation under Grant No. DMS--2424139, while the authors were in residence at the Simons Laufer Mathematical Sciences Institute in Berkeley, California, during the Spring 2026 semester.
Taylor was also partially supported by NSF grant DMS--2503113 and the Simons Foundation, and Huang was partially supported by NSF grant DMS-2005328. 

\section{Background}
Here, we briefly review the background needed for the paper and provide the reader with references where further details can be found.

\subsection{Pseudo-Anosov flows and their orbits spaces}
\label{sec:pA}
For an excellent introduction to pseudo-Anosov flows see
the recent book by Barthelm\'e and Mann \cite{barthelme2025pseudo}, where the details and nuances of the theory can be found.
Informally, a \emph{pseudo-Anosov flow} $\phi$ on a closed manifold $M$ is a flow such that there is a pair of transverse $2$-dimensional singular foliations $\FF^s$ and $\FF^u$ in $M$ with the properties that every leaf is a union of flowlines, and every two flowlines on a leaf of $\FF^s$ (resp. $\FF^u$) are asymptotic in the forward (resp. backward) direction. The singularities are a finite collection of closed orbits, each of which is locally modeled on a pseudo-hyperbolic orbit with $2n$-prongs, alternating between contracting and repelling.

We will also consider a slightly more general class of flows, called \emph{almost pseudo-Anosov flows}. We say a flow $\phi^\sharp$ is almost pseudo-Anosov if it is obtained from a pseudo-Anosov flow $\phi$ by \emph{dynamically blowing up} a collection of singular orbits. Under this operation, a singular orbit is replaced by a complex (called a \emph{blowup complex}) consisting of a finite union of \emph{blowup annuli} glued along their boundaries. The blowup complex has a transverse section that is a finite tree whose vertices correspond to closed orbits of $\phi^\sharp$, and each blowup annulus (corresponding to an edge of the tree) 
is foliated by flowlines which converge to distinct boundary orbits of the annulus in forward and backward time. See \Cref{fig:blowup}.

\begin{figure}[h]
    \centering
    \includegraphics[width= .5\textwidth]{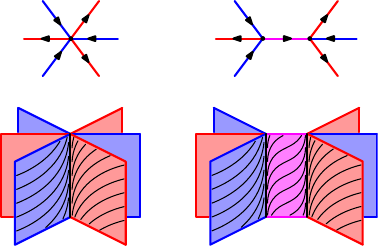}
    \caption{Figure from \cite{landry2023endperiodic} demonstrating a dynamic blowup of a 3-pronged singular orbit and its transverse cross section. Here there is a single blowup annulus.}\label{fig:blowup}
\end{figure}
 
There is also a pair of stable and unstable foliations associated to $\phi^\#$ that are transverse except along the blowup annuli, which are contained in each foliation. The two flows are related by a blowdown map $M \to M$, homotopic to the identity, that collapses each blowup complex to a single singular orbit, and semiconjugates the flow $\varphi^\sharp$ with its invariant foliations to $\varphi$ with its invariant foliations.

Almost pseudo-Anosov flows and dynamic blowups were first defined by Mosher \cite{mosher1990correction, mosher1992dynamical} and an explicit construction with many details is given by Landry--Minsky--Taylor \cite{landry2025transverse}. See these references for additional details.

\smallskip
Now let $\phi$ be an almost pseudo-Anosov flow with invariant foliations $\mc F^{s/u}$. If we lift $\phi$ to a flow $\wt{\phi}$ on the universal cover $\wt{M}$ of $M$, then the \emph{orbit space} (also called the \emph{flow space}) $\orb$ is the quotient of $\wt{M}$ by $\wt{\phi}$ with the quotient topology. It was proved by Fenley--Mosher that $\orb$ is homeomorphic to $\R^2$ \cite{fenley2001quasigeodesic} (see also \cite[Theorem 3.6]{landry2025transverse}), and we denote the projection map by $\Theta\colon \wt M \to \orb$. The singular foliations $\FF^s$ and $\FF^u$ also lift to singular foliations $\wt{\FF}^s$ and $\wt{\FF}^u$ in $\wt{M}$, and they project to transverse 1-dimensional singular foliations on $\orb$, denoted by $\orb^s$ and $\orb^u$ respectively. The action of $\pi_1(M)$ on $\wt{M}$ descends to an action of $\pi_1(M)$ on $\orb$ by homeomorphisms preserving the pair of foliations $\orb^s$ and $\orb^u$. A \emph{leaf slice} of a leaf $\ell$ in $\orb^s$ or $\orb^u$ is a properly embedded copy of $\R$ in $\ell$. When $\ell$ is periodic, a \emph{half-leaf} of $\ell$ is a complementary component of the periodic points in $\ell$. Note that a half-leaf can be either an infinite ray or a blowup segment, which is the projection of some lift of a blowup annulus. For every half-leaf $h$ of $\ell$, we also call both $\Theta^{-1}(h)$ and its image in $M$ a half-leaf of the associated foliation. In the same manner, we can also talk about leaf slices of $\wt {\mc F}^s$ and $\wt {\mc F}^u$.

\subsection{Transverse surfaces and shadows}\label{sec:transvese_shadow}

For a pseudo-Anosov flow $\phi$ and an embedded closed surface $S$ in $M$, we say $S$ is \emph{almost transverse} to $\phi$ if there is a dynamic blowup $\phi^\#$ of $\phi$ so that $S$ can be isotoped to be transverse to $\phi^\#$. We say the blowup is minimal if there is no blowup of $\phi$ transverse to $S$ up to isotopy with fewer blowup annuli. This is equivalent to the condition that $S$ intersects every blowup annulus in its core curve. If $S$ is almost transverse to $\phi$, then $S$ is incompressible \cite{mangum1998incompressible,fenley1999surfaces}, and $[S]\in H_2(M,\Z)$ is in $\cone_1^\vee(\phi)$, where $\cone_1^\vee(\phi)$ is the cone in $H_2(M, \RR)$ consisting of all the classes 
that have non-negative intersection number with any closed orbit of $\phi$. Conversely, Mosher \cite{Mos92, mosher1992dynamical} shows that any class in $\cone_1^\vee(\phi)$ is represented by an almost transverse surface, and the set of almost transverse surfaces is characterized in \cite{landry2025transverse}.

Moving forward, $\phi$ will usually denote an almost pseudo-Anosov flow that is minimally blown up to be transverse to the surface $S$. 
\smallskip

For a surface $S$ transverse to an almost pseudo-Anosov flow $\phi$ with invariant foliations $\FF^{s}$ and $\FF^{u}$, the intersections of $\FF^{s}$ and $\FF^{u}$
with $S$ give a pair of singular foliations on $S$, which we denote by $\FF_S^{s}$ and $\FF_S^{u}$, respectively. 
If we assume $M$ is hyperbolic, then by tameness \cite{agol2004tameness, calegari2006shrinkwrapping} and the covering theorem \cite{canary1996covering}, $S$ is either a virtual fiber or a quasi-Fuchsian surface. The surface $S$ is quasi-Fuchsian if and only if either one (equivalently, both) of $\FF_S^{s}$ and $\FF_S^{u}$ has a closed leaf. 
For pseudo-Anosov suspension flows this was proven by Cooper--Long--Reid  \cite{cooper1994bundles} and for general pseudo-Anosov flows by Fenley \cite{fenley1999surfaces}.
In fact, they show that $\FF_S^s$ and $\FF_S^u$ are finite foliations, i.e. foliations whose non-compact leaves all spiral into closed leaves, when $S$ is quasi-Fuchsian. The argument directly generalizes to the case where $\varphi$ is an almost pseudo-Anosov flow, and this can be seen as follows:

Suppose $S$ is transverse to the almost pseudo-Anosov flow $\phi$. In the universal cover $\wt{M}$, take any component $\wt{S}$ of the preimage of $S$. The shadow $\Omega(\wt{S})$ of $\wt{S}$ is the image of $\wt{S}$ in the orbit space $\orb$ associated to $\wt{\phi}$ under the natural projection map $\Theta \colon \wt{M}\to\orb$. The projection map is a homeomorphism from $\wt{S}$ to $\Omega(\wt{S})$, and $\pi_1(S)$ acts on $\Omega(\wt{S})$ as a covering action. The surface $S$ is quasi-Fuchsian if and only if $\Omega(\wt{S})$ is not all of $\orb$, and otherwise $S$ is a cross section of $\varphi$ (\cite[Proposition 4.5]{fenley1999surfaces}). 
Technically, this is proven only in the pseudo-Anosov case, but the argument there works in general. Indeed, the first step is to split $\mc F^s$ open along its singular leaves to produce an essential laminations transverse to $S$. The resulting lamination is the same regardless of whether $\phi$ has blowup complexes
because these are removed when splitting open the foliation.

In case where $S$ is not a cross section, it is shown in \cite[Sections 7 and 8]{landry2025transverse} that $\Omega(\wt{S})$ is bounded by leaf slices,
and each frontier leaf slice is periodic with the stabilizer in $\pi_1(S)$ isomorphic to $\Z$. If $\ell$ is a stable leaf slice in the frontier $\fr \Omega(\wt{S})$ of $\Omega(\wt S)$, a generator $
\gamma$ of $\mathrm{Stab}_{\pi_1(S)}(\ell)$ acts on $\ell$ with 
a unique fixed point $x$ having an unstable half-leaf terminating at $x$ and intersecting $\Omega(\wt{S})$. Moreover, such an unstable half-leaf is unique (\cite[Proposition 7.8]{landry2025transverse}). The intersection of this unstable half-leaf with $\Omega(\wt{S})$ is a $\gamma$--invariant line that projects to a closed component of $\FF^u_S$. 

Similarly, one can produce a closed component in $\FF^s_S$. The union $c^{s/u}$ of the closed leaves in $\FF_S^{s/u}$ are the \emph{stable/unstable multicurves} of $S$ associated to $\phi$.

\smallskip
We record an observation below that will help us identify components in $c^s$ and $c^u$. It follows from \cite{cooper1994bundles}, \cite{fenley1999surfaces} (see also \cite[Lemma 8.3]{LMT21}).

\begin{lemma}\label{lem:identify-multicurves}
    Let $M$ be an atoroidal $3$-manifold with a closed surface $S$ that is transverse to an almost pseudo-Anosov flow $\phi$. Fix a component $\wt S$ of the preimage of $S$ in $\wt M$ and let $\pi_1(S)$ denote its stablizer.

     Suppose $\alpha$ is a simple closed curve on $S$, and $\gamma\in\pi_1(S)$ is an element representing $\alpha$. If $\gamma$ fixes a stable frontier leaf of $\Omega(\wt{S})$, then $\alpha$ is homotopic to a component of $c^u$. Similarly, if $\gamma$ fixes an unstable frontier leaf of $\Omega(\wt{S})$, then $\alpha$ is homotopic to a component of $c^s$.
\end{lemma}

\begin{proof}
    Suppose $\ell$ is a stable leaf slice in $\fr \Omega(\wt{S})$ fixed by $\gamma$. Form the above discussion, there exists a unique $\gamma$-fixed point $p\in\ell$ with a unique unstable half-leaf $r$ based at $p$ and intersecting $\Omega(\wt{S})$. Then $r\cap\Omega(\wt{S})$ is also fixed by $\gamma$, projecting to a component of $c^u$ homotopic to $\alpha$. 
    
    The other case is similar.
\end{proof}

\begin{remark}
\label{rmk:outside}
As observed in \cite[Lemma 8.3]{LMT21}, 
in the statement of \Cref{lem:identify-multicurves}, one can instead assume that $\gamma$ fixes a point $p \in \orb$. The point $p$ cannot be contained in $\Omega(\wt S)$ because then $\alpha$ would be homotopic in $M$ to a closed orbit that crosses $S$. Hence, $p$ is separated from $\Omega(\wt S)$ by a unique frontier leaf $\ell$, which is necessarily fixed by $\gamma$. Moreover, $\ell$ is stable/unstable when $\Theta^{-1}(\ell)$ lies below/above $\wt S$ in $\wt M$ with respect to its coorientation induced from $\varphi$. For more details, see the refences above and the remarks following the proof of \cite[Proposition 4.3]{fenley1999surfaces}.
\end{remark}

\subsection{Curve graphs and subsurface projections}
\label{sec:curvegraph}

Let $S = S_{g,b}$ be a compact topological surface with genus $g$ and $b$ boundary components. The \emph{curve graph} $\CC(S)$ of a compact surface $S$, first introduced by Harvey \cite{Ha} and popularized by Masur--Minsky \cite{MM1}, is defined as follows. If $S$ is not an annulus, then $\CC(S)$ is a simplicial graph with vertex set $\CC^0(S)$ the isotopy classes of essential simple closed curves on $S$, and two vertices are connected by an edge if they have representatives with minimal intersection number, which is $1$ when $S=S_{1,0}$ or $S_{1,1}$, $2$ when $S=S_{0,4}$ and $0$ otherwise. We assign length one to each edge and consider $\CC(S)$ with its induced graph metric. For two essential simple closed curves $\alpha$ and $\beta$ on $S$, we define $d_{\CC(S)}(\alpha,\beta)$ to be the distance in $\CC(S)$ between their corresponding vertices. For subsets $X,Y \subset \C(S)$, we set $d_{\C(S)} (X,Y)= \diam (X \cup Y)$.

If $S$ is an annulus, we define $\CC(S)$ to have vertices the set of isotopy classes of embedded arcs connecting two boundaries of $S$, where the isotopy is relative to endpoints. Two vertices are connected by an edge if they have disjoint representatives. In this case it can be shown that $\CC(S)$ is quasi-isometric to $\Z$ \cite[Section 2.4]{MM2}.

\smallskip
When $Y$ is an essential subsurface of $S$,
there is a partially defined \emph{subsurface projection} $\pi_{Y}$ from $\C(S)$ to the power set of $\C (Y)$, which is defined as follows (see \cite[Section 2]{MM2} for the complete details).
Let $S_Y$ be the cover of $S$ corresponding to the subgroup $\pi_{1}(Y) < \pi_{1}(S)$. Fix an arbitrary complete hyperbolic metric on $S$, and consider the Gromov compactification $\overline S_Y$ of $S_Y$.
Then, given any essential multicurve $\alpha$ on $S$, 
we consider its complete preimage $\wt {\alpha}$ in $\overline S_Y$ and
define $\pi_{Y}(\alpha)$ to be the subset of (isotopy classes of) curves and arcs in $\wt \alpha$ that are essential in $Y \cong \overline S_Y$; this can be either a collection of pairwise disjoint curves or arcs, or empty in the event that $\alpha \cap Y = \emptyset$, up to isotopy. As such, $\pi_{Y}$ is technically a map from $\C(S)$ to the power set of the \emph{curve and arc graph}, but each arc can be completed to a curve using an arc along the boundary of $Y$ so that the resulting collection of curves has diameter at most $2$ \cite[Lemma 2.2]{MM2}.

Given two multicurves $\alpha, \beta \in \C(S)$, we define their $Y$-\emph{subsurface distance} to be the diameter of the projection $\pi_{Y}(\alpha) \cup \pi_{Y}(\beta)$ in $\C(Y)$:
\[
d_{\C(Y)}(\alpha, \beta) = \diam_{\C(Y)} (\pi_{Y}(\alpha) \cup \pi_{Y}(\beta)).
\]
If either $\alpha$ or $\beta$ does not intersect $Y$ essentially (i.e. they are disjoint up to isotopy), then we use the convention that $d_{\C(Y)}(\alpha, \beta) = 0$. 

\section{Flow saturated annuli and disks}
\label{sec:annuli}

Let $S$ be a closed embedded surface in $M$ transverse to an almost pseudo-Anosov flow $\phi$, and let $N$ be the cut manifold $M\cut S$. 
The boundary $\partial N$ can be decomposed as $\partial_+N \sqcup \partial_-N$ where $\phi$ flows into $N$ on $\partial_-N$ and flows away from $N$ on $\partial_+N$. We sometimes implicitly identify $\partial_\pm N$ with $S$ and use $\varphi_N$ to denote the induced semiflow on $N$.

A \emph{product annulus} $A$ in $N$ is an embedded essential (i.e. incompressible and non boundary parallel) annulus with one boundary component $\partial_+A$ on $\partial_+N$ and the other boundary component $\partial_-A$ on $\partial_-N$. Moreover, a \emph{product flow annulus} is a product annulus that is foliated by flow segments from $\partial_-A$ to $\partial_+A$. One of the main results of this section is \Cref{lem:saturated-annulus}, which says that any product annulus $A$ in $N$ 
 can either be homotoped to a product flow annulus $A_\phi$, 
 or the two boundary curves $\partial_\pm A$ of $A$ give components in $c^{u/s}$, up to isotopy in $\partial_\pm N$.

\smallskip
Let $\gamma$ be an element in $\pi_1(M)$ whose conjugacy class is represented by the core curve of a product annulus $A$, and let $\wt{A}$ be a lift of $A$ to $\wt{M}$ stabilized by the action of $\gamma$. Then the boundary of $\wt{A}$ lies on two lifts $\wt{S_-}$ and $\wt{S_+}$ 
covering $\partial_-N$ and $\partial_+N$ respectively, and $\gamma$ preserves both $\wt{S_-}$ and $\wt{S_+}$. 

The proof of \Cref{lem:saturated-annulus} requires an analysis of the shadows of $\wt{S_-}$ and $\wt{S_+}$ in $\orb$. Recall that for a subset $U$ in $\wt{M}$, the projection of $U$ to $\orb$, called its \emph{shadow}, is denoted by $\Omega(U)$. As before, the frontier of $\Omega(U)$ in $\orb$ is denoted $\fr \Omega(U)$.

\begin{lemma}\label{lem:2-shadows}
If $\Omega(\wt{S_-})$ and $\Omega(\wt{S_+})$ intersect in $\orb$, then the intersection is connected.
\end{lemma}

\begin{proof}
Any component $s$ of $(\mathrm{Fr}\,\Omega(\wt{S_-}))\cap\Omega(\wt{S_+})$ is a segment of a leaf slice $\ell$ of $\orb^{s/u}$. The slice $\ell$ separates $\orb$, and $\ell\cap\Omega(\wt{S_+})$ is connected because leaves of $\orb^{s/u}$ intersect in connected sets.
Therefore, $s$ separates $\Omega(\wt{S_+})$ into two components and $\Omega(\wt{S_-})$ can intersect only one of them. The argument works for any such $s$, so $\Omega(\wt{S_-})\cap\Omega(\wt{S_+})$ is connected.
\end{proof}

Returning to the discussion above, observe that $\gamma$ preserves $\wt{S_-}$ and $\wt{S_+}$, hence also $\Omega(\wt{S_-})$ and $\Omega(\wt{S_+})$. If $A_\phi$ is a product flow annulus in $N$ properly homotopic to $A$, then $A_\phi$ has a $\gamma$--invariant lift $\wt{A_\phi}$ connecting $\wt{S_-}$ and $\wt{S_+}$. The shadow $\Omega(\wt A_\phi)$ of $\wt{A_\phi}$ is a 
$\gamma$--invariant open arc that is properly embedded
in the disk $\Omega(\wt{S_-})\cap\Omega(\wt{S_+})$.  

Conversely, any $\gamma$--invariant properly immersed open path $\alpha$ in $\Omega(\wt{S_-})\cap\Omega(\wt{S_+})$ determines a product flow annulus $A_\phi \subset N$ that is properly homotopic
to $A$ as follows.  Let $\wt N$ be the component of the preimage of $N$ containing $\wt{S_-}$ and $\wt{S_+}$ as boundary components and set $D_\phi = \Theta^{-1}(\alpha) \cap \wt N$. This is a $\gamma$--invariant properly immersed strip in $\wt N$ and its image $A'_\phi$ in $N$ is a properly immersed annulus that is foliated by flow segments. Note that $A_\phi'$ is properly homotopic to $A$ in $N$ because they both lift to $\gamma$-invariant strips connecting $\wt{S}_+$ and $\wt{S}_-$. We claim that there is a way to resolve the self-intersections of $A'_\phi$ along flow segments to obtain a product flow annulus $A_\phi$ properly homotopic to $A$. For this, note that $A'_\phi \cap \partial_+ N$ is equal to the projection of $\alpha$ to $S$ (under the identification of $S$ with $\partial_+N$) and so it is homotopic to the simple closed curve $\partial_+A = A \cap \partial_+N$. Hence, we can resolve it along bigons and monogons to obtain a simple closed curve homotopic to $\partial_+A$. Since $A'_\phi$ intersects itself along flow segments, this resolution can be extended to a resolution of $A_\phi'$ to obtain a product flow annulus $A_\phi$ with simple boundaries (and hence embedded). 
Since $\partial_-N$ incompressible, the resolution induces surgeries along trivial bigons and monogons in $\partial A'_\phi \cap \partial_- N$, and so we have that
$\partial_-A_\phi$ is also homotopic to $\partial_-A$ in $\partial_- N$. Finally, we claim that $A_\phi'$ and $A$ are properly homotopic in $N$. This is because all the resolutions we performed can be realized through proper homotopies, so $A_\phi$ is properly homotopic to $A_\phi'$, which is further properly homotopic to $A$ as we remarked earlier.

We are now ready to prove the following lemma. Note again that we implicitly identify $\partial_+N$ and $\partial_-N$ with $S$.

\begin{lemma}[Product annuli] \label{lem:saturated-annulus}
Let $A\subseteq N$ be any product annulus. Then either $A$ is properly homotopic in $N$ to a product flow annulus, or $\partial_+A$ is a component of $c^u$ and $\partial_-A$ is a component of $c^s$, up to isotopy in $\partial N$.
\end{lemma}

\begin{proof}
If $\Omega(\wt{S_-})$ and $\Omega(\wt{S_+})$ intersect, then by \Cref{lem:2-shadows}, for any point $x\in\Omega(\wt{S_-})\cap\Omega(\wt{S_+})$, there is an arc $\alpha_0$ contained in $\Omega(\wt{S_-})\cap\Omega(\wt{S_+})$ connecting $x$ to $\gamma x$. The concatenation $\bigcup_n \gamma^n \cdot \alpha_0$ is a properly immersed path in $\Omega(\wt{S_-})\cap\Omega(\wt{S_+})$ that is $\gamma$--invariant. 
By the discussion above, this gives a product flow annulus $A_\phi$ in $N$ which is properly homotopic to $A$.

If $\Omega(\wt{S_-})$ and $\Omega(\wt{S_+})$ are disjoint, then $\Omega(\wt{S_-})$ has a periodic frontier leaf slice $\ell_-$ separating $\Omega(\wt{S_-})$ from $\Omega(\wt{S_+})$. Since both $\Omega(\wt{S_-})$ and $\Omega(\wt{S_+})$ are fixed by $\gamma$, so is $\ell_-$. The $2$-dimensional leaf $\Theta^{-1}(\ell)$ in $\wt{M}$ corresponding to $\ell_-$ must separate $\wt{S_-}$ and $\wt{S_+}$. 
Since $\wt{S_+}$ lies on the positive side of $\wt{S_-}$, 
 $\ell_-$ is an unstable leaf (see, for example, \cite[Proposition 7.8]{landry2025transverse}). By \Cref{lem:identify-multicurves}, $\partial_-A$ is homotopic to a component of $c^s$. The case for $\partial_+A$ and $c^u$ is similar.
\end{proof}

\subsection{Principal stable/unstable curves}\label{sec:principal}
Let $\omega$ be a closed orbit of $\varphi$ that is contained in $N$ and let $A^u_\omega$ be an unstable periodic half-leaf based at $\omega$, whose backwards orbits accumulate on $\omega$,
such that the closed components of $A^u_\omega \cap S$ (if any) form a subset of $c^u \subset S$. Note that every component of $c^u$ comes from such an intersection essentially by definition. The unstable curves of $A^u_\omega \cap S$ are ordered $c_1,\ldots, c_k$ by their proximity to $\omega$ and for $2\le i \le k-1$, the pair $c_i, c_{i+1}$ cobounds a subannulus of $A^u_\omega$ that is naturally a product flow annulus in $N$. The unstable curve $c_1$ cobounds a subannulus of $A^u_\omega$ with $\omega$ itself that is entirely contained in $N$; such an unstable curve is called \emph{principal} (\Cref{fig:principal}). We denote the nonempty sub multicurve of $c^u$ consisting of principal curves by $c^u_p$.

\begin{figure}[h]
    \centering
    \labellist
    \pinlabel $\omega$ at 78 0
    \pinlabel $c_1$ at 150 50
    \pinlabel $c_2$ at 150 65
    \pinlabel $c_3$ at 150 85
    \endlabellist
    \includegraphics[width= .3\textwidth]{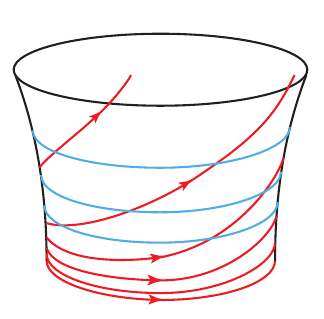}
    \caption{An unstable periodic half-leaf $A^u_\omega$ based at an orbit $\omega$ in $N$. It intersects $S$ in three closed leaves $\{c_1,c_2,c_3\} \subset c^u$, where $c_1$ is principal and the others cobound product flow annuli in $N$.}\label{fig:principal}
\end{figure}
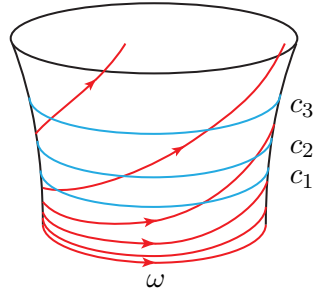

In the same manner, we define the principal stable multicurve and denote it $c^s_p$.

Finally, we note that it is clear from this discussion that each component of $c^{s/u} \subset S$ inherits a canonical orientation so that it is homotopic to $\omega$ along the annulus $A^{s/u}_\omega$. 

\begin{corollary}\label{cor:window-distance}
For any product annulus $A$ in $N$, we have
\[
d_{\CC(S)}(\partial_-A, c_p^{s})\leq 1 \quad  \text{ and } \quad d_{\CC(S)}(\partial_+A, c_p^{u})\leq 1.
\]
\end{corollary}

\begin{proof}
 By \Cref{lem:saturated-annulus}, we can assume that $A$ can be homotoped to be a product flow annulus $A_\phi$ in $N$.
We will consider the inequality for $\partial_+A$ and $c_p^{u}$, since the one for $\partial_-A$ and $c_p^{s}$ is similar. In particular, we show that if $c_1$ is a principal unstable loop, then it is disjoint from $\partial_+ A_\phi$.

By definition of a principal unstable curve, there exists a periodic orbit $\omega$ in $N$ so that $\omega$ and $c_1$ cobound a subannulus of $A^u_\omega$ that is contained in $N$.
If $\partial_+A_\phi$ intersects $c_1$ at $x$, then the backward orbit of $x$ has to be contained in $N$ and asymptotic to $\omega$. 
On the other hand, backwards orbits in $A_\phi$ run from $\partial_+A_\phi \subset \partial_+ N$ to $\partial_-A_\phi \subset \partial_-N$, giving a contradiction.

We conclude that $\partial_+A$ and $c_p^u$ have distance at most one in $\CC(S)$.
\qedhere
\end{proof}

\begin{remark}[JSJ annuli] \label{rmk:jsj}
\Cref{cor:window-distance} implies that for any product annulus $A$ in $N$, $d_{\CC(S)}(\partial_+ A, c^{u})\leq 2$ and $d_{\CC(S)}(\partial_- A, c^{s})\leq 2$,
since the curve graph diameter of $c^u$ and $c^s$ is at most one. However, if $\partial_+ A$ has an essential intersection with some $c\subset c^u$, then $c$ is not principal and hence it cobounds a product flow annulus $A_\phi$, as in the discussion preceding \Cref{cor:window-distance}. Hence, the essential annuli $A$ and $A_\phi$ cannot be homotoped off one another and so we learn that $A$ is not a JSJ annulus (i.e. not an annulus of $\mc A$) as in \Cref{sec:jsj}.
\end{remark}

We also prove an analogous result for non-product essential annuli.

\begin{lemma}\label{lem:unbalanced-annuli}
    If $A$ is any properly embedded essential annulus in $N$ with $\partial A\subset\partial_-N$ (respectively, $\partial A\subset\partial_+N$), then both components of $\partial A$ are isotopic in $\partial N$ to components of $c^s$ (respectively, $c^u$).
\end{lemma}

\begin{proof}
    We will only consider the case when $A$ has boundary in $\partial_- N$. The other case is similar. Let $\gamma$ be an element in $\pi_1(M)$ representing the free homotopic class of the core of $A$. Then $A$ has a lift $\wt{A}$ in $\wt{M}$ fixed by $\gamma$. Since $A$ is an essential annulus in $N$, the two components of $\partial\wt{A}$ lie in different lifts, denoted by $\wt{S_1}$ and $\wt{S_2}$, of $S$. We claim that shadows $\Omega(\wt{S_1})$ and $\Omega(\wt{S_2})$ are disjoint. Indeed, $\wt{A}$, $\wt{S_1}$ and $\wt{S_2}$ are contained in a lift $\wt{N}\subseteq\wt{M}$ of $N$, and on $\wt{S_1}$ and $\wt{S_2}$ the flow $\wt{\phi}$ flows into $\wt{N}$. So no flowline can intersect both $\wt{S_1}$ and $\wt{S_2}$.

    The shadow $\Omega(\wt{S_1})$ has a frontier leaf slice $\ell_1$ separating $\Omega(\wt{S_1})$ and $\Omega(\wt{S_2})$. Since $\gamma$ fixes both $\Omega(\wt{S_1})$ and $\Omega(\wt{S_2})$, it also fixes $\ell_1$ and some $p \in \ell_1$. The $2$-dimensional leaf slice in $\wt{M}$ corresponding to $\ell_1$ is on the same side of $\wt{S_1}$ as $\wt{S_2}$, that is the positive side. This means that $\ell_1$ is an unstable leaf. The unique stable periodic ray $r_1$ from $p_1$ intersecting $\Omega(\wt{S_1})$ projects to a component $\alpha_1$ of $c^s$ on $\partial_-N$, and $\alpha_1$ is homotopic to the component of $\partial A$ that lifts to $\wt{A}\cap\wt{S_1}$ (see \Cref{lem:identify-multicurves}). The same argument can be applied to $\wt{S_2}$ and the other component of $\partial A$, proving the lemma.
\end{proof}

\subsection{Transverse surfaces}\label{sec:trans_surf}
The rest of this section is devoted to studying properly embedded transverse surfaces in $N$. 

Consider a properly embedded surface $\Sigma$ in $N$ \emph{transverse} to $\phi$, which means there is an oriented branched surface structure on $\Sigma\cup\partial N$ that is transverse to $\phi$ and has no \emph{disks of contact}, 
which in this setting is a disk $D$ in $\partial N$ bounded by a component of $\partial \Sigma$ so that $D$ is on the one-sheeted side of $\partial D$.
Note that if such a disk of contact exists, then we can always use it to cap off the component of $\partial \Sigma$ producing a new transverse surface with fewer boundary components. This is sometimes called `splitting open' the disk of contact.

\begin{figure}[h!]
    \centering
    \includegraphics[width= .4\textwidth]{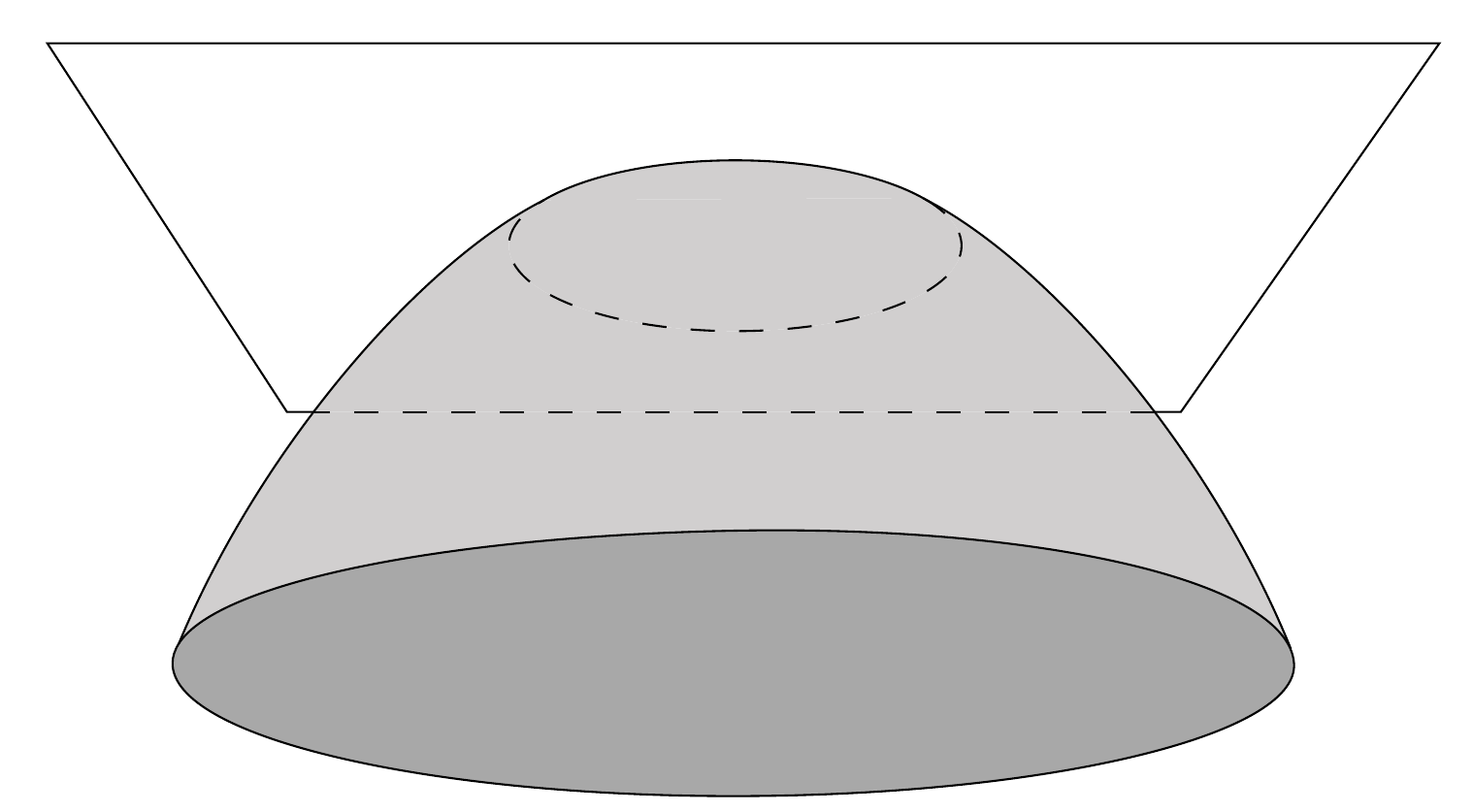}
    \caption{A disk of contact in $\Sigma \cup \partial N$.}\label{fig:disk-of-contact}
\end{figure}

For any properly embedded surface $\Sigma\subseteq N$ that is transverse to $\phi$, we can \emph{spin} $\Sigma$ around $\partial N$ to obtain a transverse (usually infinite-type) surface $L$ as follows (see also \cite[Theorem 3.7]{agol2008criteria} or \cite[Section 3.2]{landry2023endperiodic}). Take a neighborhood $\partial N\times[0,1]$ of $\partial N$ in $N$ that is foliated by flow segments such that $\partial N=\partial N\times\{0\}$. Then $L$ is the oriented cut-and-paste sum of $\Sigma$ and the collection of surfaces
\[
\bigcup_{n\in\Z_+}\partial N\times \left \{\frac{1}{n} \right \},
\]
after removing any closed components isotopic to components of $\partial N$. Such closed components only occur when $\partial \Sigma$ is nullhomologous on some component of $\partial N$. If this is the case for each component of $\partial N$, then $L$ is a closed surface.

Many properly embedded transverse surfaces can be spun to the same $L$, but we can remember the initial surface by keeping track of the induced \emph{junctures} on $L$; these are curves on $L$ associated to $\partial \Sigma$ coming from the intersection of $\Sigma$ with the parallel copies of $\partial N$. We note that with this construction, there is an induced copy of $\Sigma$ in $L$, which we call $\Sigma_L$, whose boundary is the ``innermost" collection of junctures. The condition that there are no disks of contact is equivalent to the condition that $\Sigma_L \subset L$ is an essential subsurface.

This construction is used in the next few arguments, as well as in \Cref{sec:int_bounds}.

\begin{lemma}[Transverse implies incompressible]\label{lem:incompressible}
    Any properly embedded surface $\Sigma$ in $N$ transverse to $\phi$ is incompressible.
\end{lemma}

\begin{proof}
    By the previous discussion, we can spin $\Sigma$ along $\partial N$ to obtain a surface $L$ transverse to $\phi$. Consider $N \subset M$ and let $\wt{L}$ be a component of the preimage of $L$ in the universal cover $\wt{M}$. Note that $\partial N$ is incompressible, so $\wt{L}$ is properly embedded. This implies that $\wt{L}$ separates $\wt{M}$. Indeed, let $\mathcal{N}(\wt{L})$ be a tubular neighborhood of $\wt{L}$. Then by the Mayer-Vietoris sequence for $\R^3=\mathcal{N}(\wt{L})\cup(\R^3\ssm \wt{L})$ one can see that the inclusion map $H_0(\mathcal{N}(\wt{L})-\wt{L})\to H_0(\R^3-\wt{L})$ is a bijection. Since $\wt{L}$ is 2-sided, $\mathcal{N}(\wt{L})-\wt{L}$ has exactly two components, and so does $\R^3-\wt{L}$.
  
By transversality, every flowline intersects $\wt{L}$ at most once, so the natural projection map from $\wt{L}$ to $\Omega(\wt{L})$ is a homeomorphism. As discussed above, 
$\Sigma$ can be realized as an essential subsurface $\Sigma_L \subset L$ and so it suffices to prove that $\wt{L}$ is simply connected, which is equivalent to the simple connectedness of $\Omega(\wt{L})$. 
    
    The claim that $\Omega(\wt L)$ is simply connected follows from the fact that $\Omega(\wt{L})$ is bounded by leaf slices in $\orb^s$ and $\orb^u$, which can be proved using the same proof as \cite[Proposition 4.1]{Fen09}. Alternatively, one can use the proof of \cite[Proposition 7.8]{landry2025transverse}. The main property that guarantees the generalization of both proofs is the following fact coming from the spinning construction: 
    $\wt{L}$ has an $\epsilon$-neighborhood in $\wt M$
    that can be identified with $\wt{L}\times I$ with bounded length flowlines being the $I$-fibers. Indeed, this is the basic property used in \cite[Lemma 7.4]{landry2025transverse}, which is the only place where compactness of the surface is used in the proof of \cite[Proposition 7.8]{landry2025transverse}.
 This finishes the proof of the lemma.
\end{proof}

Finally, we will prove a lemma analogous to \Cref{lem:saturated-annulus} for boundary compressing disks. Recall that a boundary compressing disk of $\Sigma$ is an embedded closed disk $D$ in $N$ with $\partial D$ divided into an arc $\alpha_\partial$ in $\partial N$ and an essential arc $\alpha_\Sigma$ in $\Sigma$. We say $D$ is a \emph{product flow disk} $E$ if its interior is foliated by flowlines, such that there is a homeomorphism $h$ from $E$ to the unit disk $\Delta=\{|z|\leq1\}$ mapping the flowlines to vertical lines, with $\alpha_\partial=h^{-1}(\{z : \mathrm{Im} z\geq0, |z|=1\})$ and $\alpha_\Sigma=h^{-1}(\{z : \mathrm{Im} z\leq0, |z|=1\})$.

\begin{lemma}[Boundary compressions]\label{lem:boundary-compressions}
Suppose $\Sigma$ is a properly embedded surface in $N$ transverse to an almost pseudo-Anosov flow $\phi$, and let $D$ be a boundary compressing disk of $\Sigma$. Then after a proper isotopy keeping $\Sigma$ transverse to $\varphi$, $D$ can be isotoped to a product flow disk.
\end{lemma}

\begin{proof}
An arc in the boundary of the boundary compressing disk $D$ lies in either $\partial_+N$ or $\partial_- N$ and for the sake of the argument we assume it lies in $\partial_+N$. The other case is similar. We let $\alpha_+ \subset \partial N$ and $\alpha_\Sigma \subset \Sigma$ be the arcs composing the boundary of $D$, and we denote the multicurve $\partial \Sigma \cap \partial_+ N$ by $j$. 

Let $\wt N$ be a copy of the universal cover of $N$ in $\wt M$, whose positive boundary is denoted by $\partial_+ \wt N$ (covering  $\partial_+N$), and let $\wt D$ be a lift of $D$ to $\wt N$. The boundary of $\wt D$ is composed of an arc $\wt \alpha_+$ in $\partial_+ \wt N$ and an arc $\wt \alpha_\Sigma$ that lies in a component $\wt \Sigma$ of the preimage of $\Sigma$. Since $\alpha_\Sigma$ is an essential arc of $\Sigma$ and $\Sigma$ is incompressible by \Cref{lem:incompressible}, $\wt \alpha_\Sigma$ joins distinct boundary components $\wt j_1$ and $\wt j_2$ of $\partial \wt \Sigma$. 
If we consider the projections to the flow space $\orb$,  $\Omega(\wt{j_1})$ and $\Omega(\wt{j_2})$ are contained in $\Omega(\partial_+ \wt N)$ and $\Omega(\wt{\Sigma})$, and in fact $\Omega(\wt{j_1})\cup\Omega(\wt{j_2})$ are in the frontier of $\Omega(\wt{\Sigma})$.

Now spin $\Sigma$ along $\partial N$ to get an infinite-type transverse surface $L$ (see the discussion at the start of \Cref{sec:trans_surf}). We can arrange the induced copy $\Sigma_L$ of $\Sigma$ in $L$ to be obtained from $\Sigma$ by isotoping a small neighborhood of $\partial\Sigma$ inside $N$ along flowlines. There is a lift $\wt{\Sigma_L}$ corresponding to $\wt{\Sigma}$, and a lift $\wt{L}$ of $L$ containing $\wt{\Sigma_L}$. Note that we have $\Omega(\wt{\Sigma})= \Omega(\wt{\Sigma_L}) \subset \Omega (\wt L)$. 

As we see in the proof of \Cref{lem:incompressible}, $\Omega(\wt{L})$ is bounded by leaf slices of the stable/unstable foliations. 
The same argument to prove \Cref{lem:2-shadows} applies here to show that $\Omega(\wt{L})\cap\Omega(\partial_+ \wt N)$ is connected. Therefore, there is an arc $\beta$ in $\Omega(\wt{L})\cap\Omega(\partial_+ \wt N)$ connecting $\Omega(\wt{j_1})$ and $\Omega(\wt{j_2})$. Although we cannot guarantee that $\beta$ is totally contained in $\Omega(\wt \Sigma)$, we can assume that neighborhoods of its endpoints are contained in $\Omega(\wt \Sigma)$ because $\Omega(\wt j_1)$ and $\Omega(\wt j_2)$ each separate $\Omega(\wt L)$. 

Now project $\beta$ to a path $\beta_+ \subset \partial_+ N$ and a path $\beta_L \subset L$ and note that $\beta_L$ flows to $\beta_+$ in the sense that there is an immersed rectangle $R = I \times [0,1]$ in $N$ so that $\beta_L$ is the image of $I \times \{0\}$, $\beta_+$ is the image of $I \times \{1\}$, and segments $t \times [0,1]$ maps into flowlines for each $t \in I$. By construction, $\beta_+  \subset \partial_+ N$ has endpoints on $j$ and is homotopic to $\alpha_+$ while keeping endpoint on $j$ (since their lifts to $\partial_+ \wt N$ join the same components of the preimage of $\wt j$). Similarly, $\beta_L$ has endpoints on $\partial \Sigma_L$ and is homotopic while keeping endpoints in $\partial \Sigma_L$ to the arc $\alpha_L \subset \Sigma_L$ that flows to $\alpha_\Sigma \subset \Sigma$ under the identification of $\Sigma_L$ with $\Sigma$. 

As in the discussion preceding \Cref{lem:saturated-annulus}, the immersed rectangle $R$ can be resolved along flow segments, this time by discarding trivial product flow annuli, giving embedded arcs $\beta_+ \subset \partial_+ N$ and $\beta_L \subset L$ that have the same properties discussed above. If $\beta_L$ were entirely contained in $\Sigma_L$, then $R$ would induce the desired product flow disk after sliding $\Sigma_L$ back up to $\Sigma$ along the flow. However, this may not be the case, and so we may need to isotope $\Sigma$ as follows.

\begin{figure}[h]
    \centering
    \includegraphics[width= .5\textwidth]{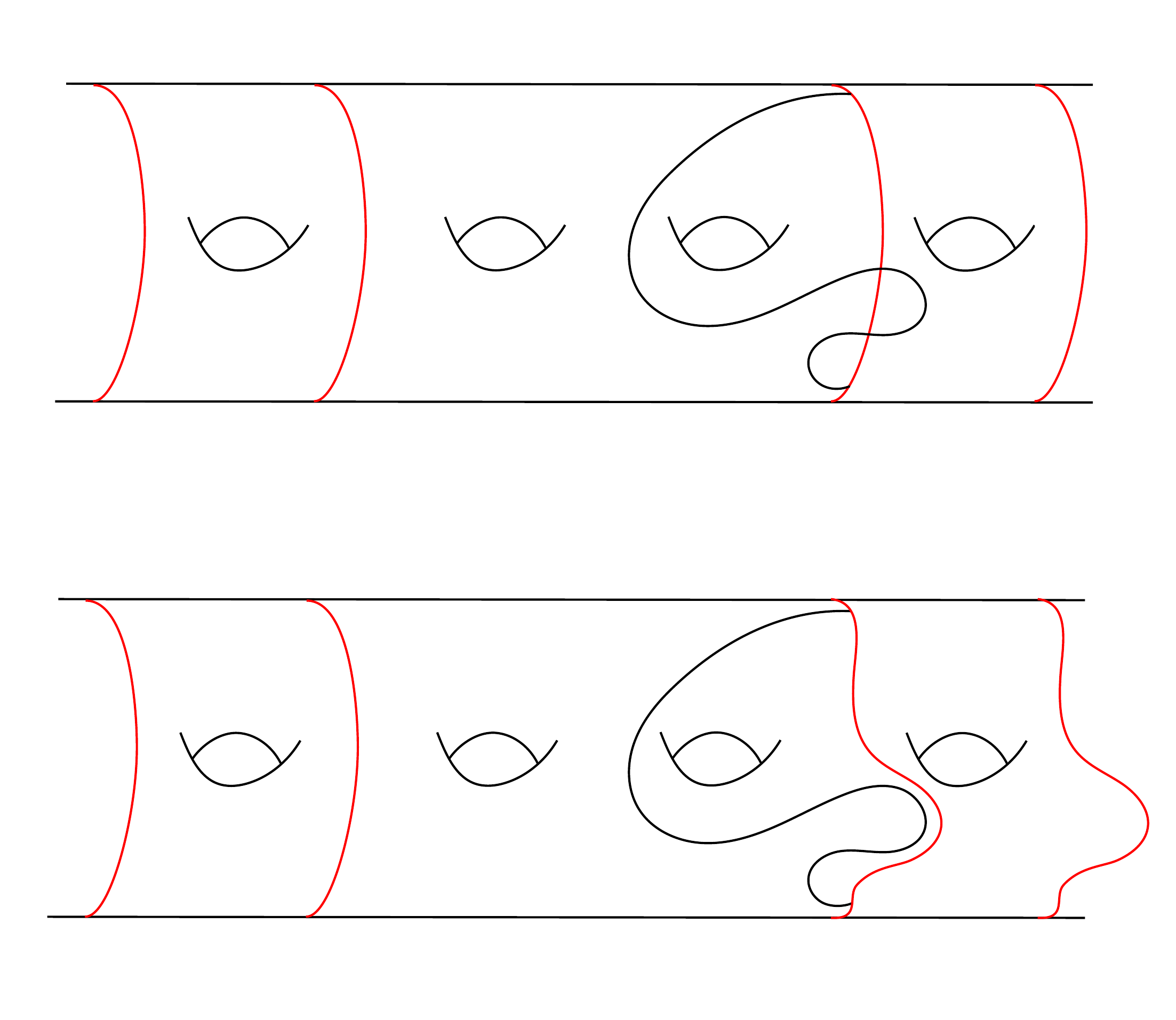}
    \caption{Adjusting $\Sigma_L$ to include $\beta_L$.}\label{fig:junc}
\end{figure}

Since $\alpha_L$ is in $\Sigma_L$ and $\beta_L$ is homotopic to $\alpha_L$ while keeping endpoints along $\partial \Sigma_L \subset L$, the arcs of
$\beta_L\cap(L\ssm \Sigma_L)$ are homotopically trivial in $L \ssm  \Sigma_L$. In other words, these arcs lie in a disjoint union of disks of $L$ that meet the boundary of $\Sigma_L$ along arcs. Let $\Sigma'_L$ be the (isotopic in $L$) subsurface obtained by adding these disks to $\Sigma_L$ (see \Cref{fig:junc}). The effect on the original $\Sigma$ is to enlarge it to a new properly embedded surface $\Sigma'$ by peeling off disks from $\partial_+ N$ (the flow images of the original disks in $L$) so that the image of $\beta_L$ is now contained in $\Sigma'$ and $R$ induces the boundary compressing disk as discussed above. 
Since $\Sigma$ can be properly isotoped to $\Sigma'$ while keeping it transverse, this completes the proof.
\end{proof}

\section{Intersection bounds}
\label{sec:int_bounds}

As before, we begin with an almost pseudo-Anosov flow $\varphi$ on a closed hyperbolic manifold $M$ and a closed transverse surface $S$. In this section, we do not assume $\phi$ to be minimally blown up to be transverse to $S$. The cut manifold $N = M \cut S$ has an induced semiflow $\phi_N$ and boundary $\partial N = \partial_+ N \sqcup \partial_- N$.

In this section, we prove the following:
\begin{proposition}\label{prop:intersection}
Suppose that $\Sigma \subset N$ is properly embedded and transverse to $\varphi$. 
If $\Sigma$ is boundary incompressible, then the number of intersections between principal
components of $c^u$ and $\partial \Sigma$ on $\partial_+N$ is bounded by a constant depending only on $|\chi(\Sigma)|$.

The same holds for $c^s$ after replacing $\partial_+N$ with $\partial_-N$. 
\end{proposition}

We begin with some setup.

The foliations $\mc F^{s/u}$ of $M$ restrict to foliations $\mc F^{s/u}_N$ of $N$ and, following \cite[Section 5]{landry2023endperiodic}, these foliations contain singular sublaminations $\Lambda^+ \subset \mc F^u_N$ and $\Lambda^- \subset \mc F^s_N$ defined as follows. We remark however that our notation is different from that of \cite{landry2023endperiodic}, where $\Lambda^\pm$ are denoted $W_N^\pm$.

As a subspace of $N$, $\Lambda^+$ is the union of all $\varphi_N$-orbits which do not meet $\partial_-N$, and $\Lambda^-$ is the union of all $\varphi_N$-orbits which do not meet $\partial_+N$. Hence, the components of the intersection $\Lambda^+ \cap \Lambda^-$ are exactly the orbits of $\varphi$ that are contained in $N$. In fact, as in \cite[Proposition 5.6]{landry2023endperiodic}, $\Lambda^\pm$ are (singular) sublaminations of $\mc F^{u/s}_N$. The leafwise description of $\Lambda^\pm$ is given by \cite[Lemma 5.4]{landry2023endperiodic}: if $H$ is a leaf $\mc F^u_N$, then $H \cap \Lambda^+$ is either empty or all of $H$, after possibly removing any periodic open half-leaves of $H$ 
that meet $\partial_- N$. (Note that each such half-leaf is the intersection of $H$ with a blowup annulus whose orbits approach the periodic orbit of $H$ in forward time.)
Moreover, each leaf of $\Lambda^+$ is of this form.
The same leafwise description holds for $\Lambda^-$, after replacing $u$ with $s$ and $+$ with $-$.
\smallskip

Let $c$ be a component of $c^u$ on $\partial_+ N$. Then either $c$ is a component of $\Lambda^+ \cap \partial_+ N$ or there is a properly embedded annular leaf $A$ of $\mc F^u_N$ for which $\partial_+ A = c^u$ and $\partial_-A \subset \partial_ -N$ where flow segments in $A$ go from $\partial_- A$ to $\partial_+ A$. The annulus $A$ is obtained by flowing points on $c$ backwards into $N$; since these points are not in $\Lambda^+$, they each exit through $\partial_- N$. Hence, referring to \Cref{sec:principal}, the closed leaves of $\Lambda^+ \cap \partial_+ N$ are precisely $c_p^u$, the principal unstable multicurve. 

Similarly, $c_p^s$ is the sublamination of closed leaves of  $\Lambda^- \cap \partial_- N$.

\smallskip
We classify the closed leaves of $\Lambda^+ \cap \partial_+ N$ as follows:

\begin{lemma}[Principal (un)stable multicurves] \label{frontier_leaves}
Each closed leaf of $\Lambda^+ \cap \partial_+ N$ is the boundary of a periodic half-leaf of $\Lambda^+$ whose closed orbit lies in a frontier leaf of $\Lambda^-$. Moreover, this association induces a bijection between 
the closed leaves of $\Lambda^+ \cap \partial_+ N$ and the isolated sides of frontier leaves of $\Lambda^-$. 
\end{lemma}

Here, we recall that a frontier leaf of $\Lambda^\pm$ is one that is isolated to at least one of its sides.

An analogous statement also holds for closed leaves of $\Lambda^- \cap \partial_- N$.

\begin{proof}
Let $c$ be a closed leaf of $\Lambda^+ \cap \partial_+ N$. By construction, $c$ is contained in a leaf $W$ of $\mc F^u_N$
and the backward orbits starting at points along $c$ converge to a closed orbit $\omega\subset W$. Moreover, these backward orbits do not exit $N$ (by definition of $\Lambda^+$) and so we conclude that there is a periodic half-leaf $\ell$ of $\Lambda^+$ such that $\omega$ is the closed orbit of $\ell$ and $\ell \cap \partial_+N = c$. 

Let $\ell^-$ be the leaf of $\Lambda^-$ containing $\omega$. Since orbits of $\ell$ other than $\omega$ cross $c$, and hence exit $N$, in forward time, we have that $\ell$ lies to an isolated side of $\ell^-$. This makes $\ell^-$ a frontier leaf of $\Lambda^-$. 

In the other direction, let $\ell^-$ be a frontier leaf of $\Lambda^-$. By \cite[Proposition 5.6]{landry2023endperiodic}, $\ell^-$ is periodic, i.e. contains a closed orbit $\omega$. (The statement there is for a special case, but the argument holds in general. One could also directly appeal to the more general statement in  \cite[Proposition 8.7]{landry2025transverse} to obtain the conclusion.) Hence, we can let $\ell$ be the half-leaf of $\Lambda^+$ that contains $\omega$ and begins on the  
isolated side of $\ell^-$. Since all orbits of $\ell$ starting near (but not on) $\omega$ exit $\partial_+N$, we have that $\ell \cap \partial_+N$ is a closed component $c$ of $\Lambda^+ \cap \partial_+ N$. This completes the proof.
\end{proof}

We now fix a boundary incompressible, properly embedded surface $\Sigma \subset N$ that is transverse to $\phi$, as in \Cref{prop:intersection}. Note that $\Sigma$ is naturally cooriented by $\varphi$ and so $\partial \Sigma$ has an induced coorientation in $\partial N$.

We will be interested in understanding the intersections $\partial \Sigma \cap c^{u} \subset \partial_+ N$ and $\partial \Sigma \cap c^{s} \subset \partial_- N$, and so we assume that the surface $\Sigma$ is chosen in its isotopy class, among surfaces transverse to $\varphi$, as to minimizes the number of intersections between $\partial \Sigma$ and  $c^{u/s}$ as curves of $\partial_\pm N$.
In fact, from 
\Cref{lem:int_sig}
we will conclude that after this isotopy, $\partial \Sigma$ is in minimal position with respect to $c^{s/u}$, but we will not need this fact a priori.

To simplify terminology we make the following definitions. Suppose that $c$ is a component of $c^u$ (the $c^s$ case is similar).  If $c \subset \Lambda^+ \cap \partial_+ N$, i.e. $c$ is principal, let $H \subset \Lambda^+$ be the periodic half-leaf with $c = H \cap \partial_+ N$, and let $\omega_c$ be the closed orbit of $\phi_N$ that cobounds $H$ with $c$.
Otherwise, let $H$ be the product flow annulus that $c$ cobounds. In either case, we call $H$ the \emph{supporting annulus} for $c$. 

\begin{lemma}[Intersecting $\Sigma$] \label{lem:int_sig}
Let $c$ be a component of $c^{s/u}$ and let $H$ be its supporting annulus. If $\partial \Sigma$ intersects $c$, then $H \cap \Sigma$ consists of arcs that run between distinct boundary components of $H$. Hence, if $c$ is principal, then $\Sigma$ intersects $\omega_c$.

Moreover, $\partial \Sigma$ is in minimal position with respect to $c^{s/u}$.
\end{lemma}

\begin{proof}
Consider the transverse intersection $H \cap \Sigma$, which by assumption contains arcs meeting $c$. For the first statement, it suffices to show that there are no arcs with both ends on $c$. Supposing otherwise, let $\alpha_\Sigma$ be an arc of intersection with both endpoints on $c \subset \partial_+ N$ that is closest to $\partial_+ N$, in the sense that it cobounds a subdisk $D \subset H$ with an arc of $c$ along $\partial_+N$. 

If $\alpha_\Sigma$ is an essential arc in $\Sigma$, then $D$ is a product flow disk, contradicting that $\Sigma$ is boundary incompressible. Otherwise, $\alpha_\Sigma$ cobounds a disk $D_\Sigma \subset \Sigma$ such that $\partial D_\Sigma \ssm \alpha_\Sigma \subset \partial \Sigma$. Hence replacing $D_\Sigma$ with $D$ produces a new properly embedded surface $\Sigma'$ that is isotopic to $\Sigma$ by irreducibility of $N$. Moreover, since $D$ is foliated by flow segments, after a small perturbation, $\Sigma'$ is transverse to $\varphi_N$ and has two fewer intersection points with $c$ than $\Sigma$ does. This contradicts our assumption on $\Sigma$ and establishes that $H \cap \Sigma$ consists of arcs that run between distinct boundary components of $H$. When $H$ is a half-leaf of $\Lambda^+$, this implies that $\Sigma$ crosses $\omega_c$.

We now use this to show that $\partial \Sigma$ is in minimal position with $c^{s/u}$. Toward a contradiction, suppose that $\partial \Sigma$ cobounds a bigon $B$ with a component $c$ of (say) $c^u$ on $\partial_+ N$. If $H$ is the supporting annulus for $c$, then the argument above shows that $c$ cannot be principal since $\partial \Sigma$ intersects $c$ both positively and negatively (and $\Sigma$ has only positive intersections with $\omega_c$). Hence, $H$ is product flow annulus and each arc of the intersection $\Sigma \cap H$ runs between distinct boundary components of $H$. The intersections arcs determined by the bigon $B$ cobound a strip $s$ in $H$ and if these arcs also determine a strip $D$ in $\Sigma$, then one can do a similar surgery as above to replace $D$ with $s \subset H$ and perturb to remove two intersections of $\partial \Sigma$ with $c^{u}$ on $\partial_+N$. Otherwise, the union $s \cup B$ can be pushed off $\partial_+ N$ to produce a boundary compressing disk for $\Sigma$. This contradicts the assumption that $\Sigma$ is boundary incompressible and completes the proof.
\qedhere

\end{proof}

\smallskip
Let $L$ be the transverse infinite-type surface obtained by spinning $\Sigma$ around $\partial N$ (see the discussion preceding \Cref{lem:saturated-annulus}). Since $L$ is transverse to $\varphi$, there are induced (singular) laminations $\lambda^\pm =  \Lambda^\pm \cap L$.
Essentially by construction, the singular lamination $\lambda^+$ is supported on the complement of the open set $\mc U^-$ of points of $L$ whose backwards orbits cross $\partial_-N$. Similarly, $\lambda^-$ is supported on the complement of the open set $\mc U^+$ whose forward orbits cross $\partial_+ N$. The subsets $\mc U^\pm \subset L$ are called the \emph{positive/negative escaping sets}, respectively.

A half-leaf of $\lambda^\pm$ (also called a \emph{ray} in $L$)
is said to be  \emph{periodic} if it is based at the intersection of $L$ and a closed orbit of $\phi_N$. (Abusing terminology, we call the intersection of $L$ with a closed orbit of $\phi_N$ a \emph{periodic point} of $L$.)
Each periodic ray is either recurrent or it exits compact sets of $L$, in which case we call it \emph{escaping}. We call a periodic escaping ray of $\lambda^\pm$ that has a well-defined first return map under $\varphi$ a \emph{principal ray}.
The next lemma ties the structure of principal rays in $\lambda^\pm$ to periodic half-leaves of $\Lambda^\pm$ that determine the principal unstable/stable curves.

\begin{lemma}[Principal rays of $\lambda^{\pm}$] \label{escaping_leaves}
If $\ell$ is a periodic half-leaf of $\Lambda^+$ such that $\ell \cap \partial_+N$ is a a closed leaf, i.e. a principal unstable curve, then $\ell \cap L$ is a (possibly empty) collection of principal rays
$\lambda^+$. 

Conversely, if $l$ is a principal ray
$\lambda^+$, 
then its flow saturation in $N$ is a periodic half-leaf of $\Lambda^+$ meeting $\partial_+N$ in a principal unstable curve.
\end{lemma}

In particular, if $l$ is a principal ray of $\lambda^+$ based at the periodic point $p \in L$, then $l \ssm p \subset \mc U^+$.

\begin{figure}[h]
    \centering
    \includegraphics[width= .4\textwidth]{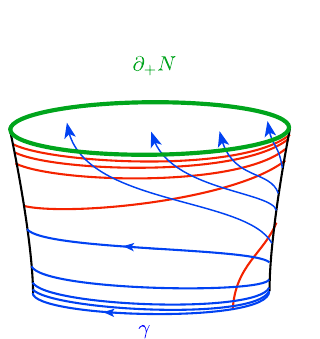}
    \caption{A figure from \cite{landry2023endperiodic} showing a periodic leaf of $\Lambda^+$ that intersects $\partial_+ N$ in a closed curve (a principal unstable curve), and the periodic escaping ray of $\lambda^+$ coming from its intersection with $L$.}\label{fig:periodic-half-leaf}
\end{figure}

\begin{proof}
For the first statement, suppose that $\ell$ is a periodic half-leaf of $\Lambda^+$ such that $\ell \cap \partial_+ N$ is closed. Then $\ell$ is the support annulus for $c = \ell \cap \partial_+ N$ and the claim follows from \Cref{lem:int_sig}

For the converse statement, let $l$ be a principal ray of $\lambda^+$ and let $\ell$ be the half-leaf of $\Lambda^+$ containing $l$. Since the ray $l$ is contained in $\lambda^+$, the first return map $l \to l$ fixes its periodic point and is expansive. Because $l$ is escaping, the forward orbit of every point exits the positive ends of $L$ and hence limits to a point on $\partial_+ N$. From this, we see that $\ell \cap \partial_+N$ is a closed leaf of $\Lambda^+ \cap \partial_+ N$, as required.
\end{proof}

The last lemma is the main way of controlling intersections.

\begin{lemma} \label{one_int}
With $\Sigma$ and $L$ as above, each principal ray of $\lambda^+\subset L$ intersects $\partial\Sigma_L$ no more than once. 
\end{lemma}

\begin{proof}
Let $l$ be a principal ray of $\lambda^+$ based at the periodic point $p \in L$. Clearly, $p \in \Sigma_L$ and suppose that $l$ intersects $\partial \Sigma_L$ more than once. Since $\partial \Sigma_L$ separates $L$, $l$ crosses $\partial \Sigma_L$ both positively and negatively. Hence, $\partial \Sigma$ crosses the principal unstable curve $c$ associated to $l$ (as in \Cref{escaping_leaves}) both positively and negatively. This contradicts
\Cref{lem:int_sig} because $\Sigma$ is positively transverse to $\varphi$.
\end{proof}

With these tools at hand, we now turn to the proof of the proposition.

\begin{proof}[Proof of \Cref{prop:intersection}]

Let $c$ be a principal component of $c^u$ that intersects $\partial \Sigma$. 
Then $c$ is a component of $\Lambda^+ \cap \partial_+ N$ and we can apply \Cref{frontier_leaves} and \Cref{escaping_leaves}. In particular, let $\ell^+$ be the periodic half-leaf of $\Lambda^+$ with $c = \ell^+ \cap \partial_+ N$.
Then, the intersections of $c$ with $\partial \Sigma$ are in bijection with the intersections between $\ell^+ \cap L$ and $\partial \Sigma_L$ along $L$ (the bijection given by flowing $\partial \Sigma_L$ to $\partial \Sigma$).  Moreover, \Cref{escaping_leaves} implies that the intersection $\ell^+  \cap L = \{ l^+_1, \ldots, l^+_k\}$ is a collection of principal
rays such that $\ell^+$ is the flow saturation of each $l^+_i$. Since each $l^+_i$ intersects each component of $\partial\Sigma_L$ at most once (\Cref{one_int}), it suffices to bound $k$ in terms of $\chi(\Sigma)$.

From \Cref{frontier_leaves}, we see that there is a (periodic) frontier leaf slice $\ell^-$ of $\Lambda^-$ whose closed orbit agrees with that of $\ell$. Hence, $\ell^- \cap L = \{l^-_1, \ldots l^-_k\}$ is a collection of frontier leaf slices of $\lambda^-$ so that the initial point of $l^+_i$ is also the periodic point of $l^-_i$. It now suffices to bound the number of periodic frontier leaf slices of $\lambda^-$.

For a frontier leaf slice $l$, let $l_\Sigma$ be the component of $l \cap \Sigma_L$ that contains the periodic point of $l$. Here, we are using the fact that all periodic points of $L$ lie in $\Sigma_L$.
We will show that at most two frontier arcs of the form $l_\Sigma$ can be parallel in $\Sigma$. 
Otherwise, we can find
 two such frontier slice leaves $l$, $l'$ of $\lambda^{-}$, such that the arcs 
  $l_\Sigma$ and $l'_\Sigma$ 
cobound a strip $s$ whose interior is contained in $\mc U^+$. Note that the two other sides of $s$ lie along arcs of $\partial_- \Sigma_L$, where $\partial_- \Sigma_L$ are the components of $\partial \Sigma_L$ that correspond to the components of $\Sigma \cap \partial_- N$. 
 
Let $p$ be the periodic point of $l$. Since $p \in l_\Sigma$, it lies along the boundary of $s$ and so there is a principal
ray $l^+$
of $\Lambda^+$ based at $p$ and directed into $s$. That $l^+$ is principal follows from  \Cref{frontier_leaves} and \Cref{escaping_leaves}; referring to the beginning of the proof, if $l = l^-_i$, then $l^+ = l^+_i$.
Since $l^+$ does not cross $\partial_- \Sigma_L$, it must cross $l'$ along the boundary of $s$ at a point $p' \in l'$. On the one hand, the backwards orbit of $p'$ (along the flow saturation of $l^+$) must meet $L$ in the interior of $s$. Since the interior of $s$ is contained in $\mc U^+$, this means that the forward orbit along $p'$ exits $\partial_+N$. But this contradicts that $l'$ is a leaf of $\lambda^-$ and so the forward orbit through any of its points stays in $N$. 
 
Since there is an upper bound on the number of disjoint nonparallel essential arcs in $\Sigma$ depending only on $\chi(\Sigma)$, the proof is completed by showing that each $l_\Sigma$ is essential. Otherwise, $l_\Sigma$ cobounds a disk $D \subset \Sigma_L$ with an arcs $a \subset \partial_-\Sigma_L$. But since $l_\Sigma$ is an arc of a leaf of $\lambda^-$ containing a periodic point $p$, the leaf $l^+$ of $\lambda^+$ through $p$ must contain a ray $r$ based at $p$ that meets the interior of $D$. This, however, is a contradiction: $r$ cannot leave $D$ through $l_\Sigma$ because that would create a bigon between leaves of $\lambda^\pm$, and cannot cross $a$ because $\lambda^+$ is disjoint from $\mc U^-$.
\end{proof}

\section{Geometric arguments: broken windows and bounded length curves}
\label{sec:geom}

Before we state the main result of this section, we give a brief overview of JSJ theory of 3-manifolds and pleated surfaces in hyperbolic 3-manifolds.

\subsection{JSJ-decomposition}\label{sec:jsj}
 We refer the readers to \cite{JacoShalen} and \cite{JohJSJ} for the complete theory of JSJ-decompositions of 3-manifolds. The formulation we use here is from \cite{Agolminvol} (see also \cite[Section 2.2]{brock2020windows} which uses the same terminology that is used here).
 For our purpose, consider a compact orientable irreducible
 $3$-manifold $N$ with non-empty incompressible boundary, such that the interior of $N$ admits a hyperbolic structure. By the theory of JSJ decomposition, there exists a canonical collection $\A$ of disjoint essential nonparallel annuli such that
\begin{enumerate}
    \item the components of $N\cut\A$ are $I$-bundles, pared solid tori, and acylindrical pared 3-manifolds;
    \item any immersed essential annulus in $N$ can be homotoped into a component of $N \cut \A$ that is an $I$-bundle or a pared solid torus.
\end{enumerate}

The \emph{window} $W(N)$ of $N$, defined by Thurston, is a submanifold of $N$ consisting of all the $I$-bundle components, and thickened neighborhood of components of $\A$ that are boundaries of the pared solid tori but not adjacent to the $I$-bundles. Then all immersed essential annuli can be isotoped into $W(N)$. The \emph{window surface} $\partial_wN$ is a subsurface of $\partial N$ given by $W(N)\cap\partial N$. The \emph{window frame} $\WF(N)$ is defined to be the boundary of $\partial_wN$. We say that a component $X$ of $\partial N$ is \emph{opaque} if it does not meet $\partial_wN$. This is equivalent to say that no essential annulus in $N$ has a boundary component in $X$.

We need the following fact, which is surely well-known to experts.

\begin{lemma}\label{lem:opa}
Let $N$ be a compact irreducible atoroidal $3$-manifold with incompressible boundary, and let $X$ be an opaque component of $\partial N$. Then if a loop $\gamma$ in $N$ is such that $\gamma^k$ is homotopic into $X$ for some $k\neq 0$, then $\gamma$ is homotopic into $X$.

Consequently, an essential simple loop in $X$ is \emph{indivisible} in $N$, i.e. not homotopic to a proper multiple of a loop in $N$.
\end{lemma}

\begin{proof}
Let $x \in X$ and let $l \subset X$ be a loop at $x$ representing $\gamma^k$. If we represent $\gamma$ by a loop in $N$ based at $x$ (without changing its name), then the relation $\gamma \cdot l \cdot \gamma^{-1} = l$ in $\pi_1(N,x)$ determines map $(A, \partial A) \to (N, X)$ from an annulus $A$ such that each component of $\partial A$ maps to $l$. Since $X$ is opaque, this map from $A$ can be properly homotoped into $X$. This produces a homotopy from $\gamma$ to a loop in $X$ and completes the proof of the first claim.

For the second claim, if $l$ is a simple loop in $X$ that is homotopic to $\gamma^k \subset N$ for some $k\neq 0$, then we already know that $\gamma$ is homotopic into $X$. Hence, assume $\gamma \subset X$. Then the homotopy from $l$ to $\gamma^k$ is again given by a map of pairs $(A, \partial A) \to (N, X)$ which, as above, is properly homotopic into $X$.  But since $X$ is an orientable surface, we must have $k = \pm 1$ and the proof is complete.
\end{proof}

Finally, the following theorem is due to Thurston \cite[Theorem 0.1]{thurston1998hyperbolic3}. An alternative proof, where the dependence only on $\partial N$ is stated explicitly, is given in the appendix of \cite{brock2016bounded}.

\begin{theorem}[Thurston's Only windows break] \label{th:OWB}
Let $N$ be a compact irreducible atoroidal $3$-manifold with incompressible boundary. Then there
exists a constant $C_1\ge 0$ depending only on $\vert \chi(\partial N) \vert$, so that for any
complete hyperbolic structure on $\intr N$,
the length of the window frame is at most $C_1$.
\end{theorem}

\subsection{Pleated surfaces}\label{sec:pleated}
A \emph{pleated surface} in $N$ is a map $f$ from a finite-type surface $S$ to the complete hyperbolic manifold $\mr N$, where $S$ is equipped with a hyperbolic metric $\sigma$ and a geodesic lamination $\lambda$, with the following properties: $f$ preserves the lengths of paths, 
maps each leaf of $\lambda$ to a geodesic in $\intr N$, and is totally geodesic in the complement of $\lambda$. We allow $S$ to have boundary, in which case $\lambda$ is required to contain $\partial S$ and $\sigma$ is a hyperbolic metric on $S$ with geodesic boundary. We say $\sigma$ is the \emph{induced hyperbolic metric} on $S$, and $\lambda$ is the \emph{pleating locus} of $f$. For any continuous map $\rho \colon S\to N$, a pleated surface $f \colon (S,\sigma)\to N$ homotopic to $\rho$ is called a \emph{pleated realization} of $\rho$. Given $\rho$ (which in practice will be the inclusion map), we say a geodesic lamination $\lambda$ on $S$ is \emph{realizable} if there is a pleated realization of $\rho$ whose pleating locus contains $\lambda$. If $\rho$ is $\pi_1$-injective, any multicurve on $S$ is realizable. See \cite{canary1987notes} for more details.

\medskip

The next proposition is the main result of this section, generalizing \cite[Lemma 4.1]{field2025lower} who establish the bound for the (unique) totally geodesic hyperbolic structure on $N$, in the case where $N$ is acylindrical.

\begin{proposition}[Bounded junctures]\label{prop:length-bound}
Let $N$ be a compact irreducible atoroidal $3$-manifold with incompressible boundary, and let $\Sigma$ be an incompressible, boundary incompressible properly embedded surface in $N$. Suppose $\alpha$ is a component of $\partial\Sigma$ on an opaque component $X\subseteq\partial N$. Then there is a constant $C_2\ge 0$, depending only on $|\chi(\partial N)|$ and $|\chi(\Sigma)|$, so that for any complete hyperbolic structure on $\intr N$,
$\ell_N(\alpha) \le C_2$.
\end{proposition}

\begin{proof}
Fix a pleated realization of $\partial N$ and $\Sigma$ in $N$ so that the multicurve $\partial\Sigma$ is in the pleated loci of both $\partial N$ and $\Sigma$. Since both pleated maps take $\partial\Sigma$ to its geodesic representative in $N$, we can combine the two maps to obtain a single map $g\colon\partial N\cup_{\partial \Sigma}\Sigma\to \intr N$. We claim that $g|_{X\cup\Sigma}$ is homotopic to the identity inclusion. Indeed, since $N$ is atoroidal and $X$ is opaque, there is a unique way up to homotopy to homotope curves in $\partial\Sigma\cap X$ to their geodesic representatives (see \Cref{lem:opa}). Both $X$ and $\Sigma$ are homotopic to their images under $g$, so we can combine the homotopies to get a homotopy from $X\cup\Sigma$ to $g(X\cup\Sigma)$. Alternatively, one can also more carefully construct $g$ to be homotopic to the identity on all of $\partial N\cup\Sigma$ by first realizing $\partial N$ as pleated surfaces and realizing $\Sigma$ in the right homotopy class as a simplicial hyperbolic surface (as in \cite{canary1996covering}) with all vertices on $\partial \Sigma$, and then spin the vertices about $\partial \Sigma$ to make $\Sigma$ pleated (as in \cite[Section 8.39]{Thurston}). As we will see later, the reason we need $g$ to be homotopic to the identity on the union $X\cup \Sigma$, instead of just on $X$ and $\Sigma$ individually, is that we would like an immersed boundary compressing disk for $g(\Sigma)$ and $g(X)$ to give an immersed boundary compressing disk for $\Sigma$ and $X$.

The pleated map $g$ induces a hyperbolic structure $\tau$ on $\partial N$, and a hyperbolic structure with geodesic boundary $\sigma$ on $\Sigma$. These are the metrics we will use when talking about $\partial N$ and $\Sigma$. We denote 
the unique geodesic representative of any closed curve $\gamma$ in $N$ by $\gamma^*$. Note that in these fixed metrics $\partial \Sigma$ is geodesic on both $\Sigma$ and $\partial N$.

\smallskip
Let $\epsilon_0$ be the Margulis constant for $\HH^3$, and recall that $\alpha$ is from the statement of \Cref{prop:length-bound}. The $\epsilon$--thin part of $N$, for $\epsilon \le \epsilon_0$, is denoted by $N^{\le \epsilon}$.

\begin{claim}\label{claim:thin-part}
There exists $R\ge 0$ and $0<\epsilon_2<\epsilon_0$, depending on the topology of $\Sigma$ and $\partial N$, such that either we have $\ell_N(\alpha^*)\leq 2R$ or $\alpha^*$ is disjoint from $N^{\leq\epsilon_2}$.
\end{claim}

\begin{proof}[Proof of Claim]
 First, we need the following observation by Thurston \cite{thurston1986hyperbolic}.

\begin{lemma}[Thurston]\label{lem:pleated-thin}
For any topological surface $S$, there exists a constant $\epsilon_1<\epsilon_0$ depending on the topology of $S$, such that for any closed pleated surface $f \colon S\to N$, we have $f^{-1}(N^{\leq\epsilon_1})\subseteq S^{\leq\epsilon_0}$.
\end{lemma}

Fix a choice of $\epsilon_1$ satisfying \Cref{lem:pleated-thin} for $\partial N$, and fix a constant $R$ such that $\cosh(R)>1-\chi(\Sigma)$. Then there exists a uniform $\epsilon_2<\min\{\epsilon_1,R\}$ such that $d_N(N^{\le\epsilon_2},N^{>\epsilon_1})>2R$ (see \cite[Lemma 6.1]{minsky1999classification}). 
Suppose $\alpha^*$ intersects an $\epsilon_2$-Margulis tube $\TTT_{\epsilon_2}(c)$ in $N^{\leq\epsilon_2}$, where $c$ is an oriented core curve in $N$. For contradiction, assume $\alpha^*$ has length greater than $2R$.

Since $\epsilon_2<R$, the curve $c$ is not homotopic to $\alpha$. Let $x$ be a point in $\alpha^*\cap\TTT_{\epsilon_2}(c)$ and let $x_1$ be a point in $\alpha\subset\partial\Sigma$ such that $g(x_1)=x$. Consider the $R$-neighborhood $B_R(x_1)$ of $x_1$ in $\Sigma$.
By our choice of $R$, the half-disk $B_R(x_1)$ cannot be embedded in $\Sigma$ because otherwise the area of $B_R(x_1)$ is greater than the area of $\Sigma$. Therefore, either there exists an essential arc on $\Sigma$ with length $\leq R$ connecting $x_1$ to $\partial \Sigma$, or we can find a geodesic loop starting and terminating at $x'$ with length $<2R$.
 If this loop is essential, we are back to the first case, except that we have an essential arc of length $\leq2R$. If this loop is homotopic to $\alpha$, then we have a contradiction to the assumption that $\ell_\sigma(\alpha)=\ell_N(\alpha^*)>2R$. 

Let $a$ denote the essential arc of $\Sigma$ containing $x_1$ that we just produced and let $x_2$ be the other endpoint of $a$. Note that both $x_1$ and $x_2$ are contained in $\partial \Sigma \subset \partial N$. The path $g(a)$ is contained in $\TTT_{\epsilon_1}(c)$, by our choice of $\epsilon_2$. Then we have that both $x_1$ and $x_2$ are contained in $g^{-1}(\TTT_{\epsilon_1}(c))\cap \partial N \subseteq (\partial N)^{\leq\epsilon_0}$ by \Cref{lem:pleated-thin}. If $x_1$ and $x_2$ are in two different $\epsilon_0$-thin tubes of $\partial N$, then (up to taking multiples) the core curves of these components are non-homotopic in $\partial N$ and are homotopic to $c$ in $N$, 
contradicting the fact that $X$ is opaque. Hence, $x_1$ and $x_2$ are contained in the same component $T$ of $X^{\leq\epsilon_0}$. Pick an oriented core curve $c'$ of $T$ homotopic to $c$ (using \Cref{lem:opa}), and take any arc $b$ in $T$ connecting $x_1$ and $x_2$. The closed curve $g(a\cup b)$ is contained in $N^{\leq\epsilon_0}$, so it is homotopic to $c^n$ for some $n$. If we let $b'$ be the arc in $T$ obtained by winding $b$ around $c'$ for $-n$ times, then $a\cup b'$ is trivial. In other words, the essential arc $a\subset \Sigma$ is homotopic to the arc $b' \subset \partial N$ rel endpoints. Note that here we implicitly use that $g|_{X\cup\Sigma}$ is homotopic to the identity. Although the homotopy from $a$ to $b'$ is not necessarily a boundary compression, because the image of the homotopy might be an immersed disk intersecting $\Sigma\cup\partial N$ in its interior, a standard argument allows us to remove inessential intersections and take the innermost essential boundary compressing disk, contradicting the boundary incompressibility of $\Sigma$. 
\end{proof}

By elementary hyperbolic geometry, for any $L>0$ there exists a constant $D(L)\ge 0$ so that if $\ell_{\sigma}(\alpha)>D(L)$, then $\Sigma$ contains a rectangular band $B$ with a pair of opposite sides $\eta_1\subseteq\alpha$ and $\eta_2\subseteq\partial \Sigma$ of length $L_\eta>L$, and of width $<\epsilon_2/8$. This constant $D(L)$ will depend on $L$, and also on $\epsilon_2$ and the topology of $\Sigma$. Note that $\eta_1$ and $\eta_2$ are also geodesic segments on $\partial N$.
Up to enlarging $D(L)$, we may assume that $D(L)$ is always greater than $2R$. Then \Cref{claim:thin-part} shows that $g(\eta_1)$ is contained in $N^{>\epsilon_2}$. Since $g(\eta_2)$ stays in the $\epsilon_2/8$-neighborhood of $g(\eta_1)$, there exists $\epsilon_3>0$ such that $g(\eta_2)$ is in $N^{>\epsilon_3}$.
We also assume $\epsilon_3$ is smaller than $\epsilon_2/4$.

We now apply a volume argument from the appendix of \cite{brock2016bounded}. More precisely, we parameterize $\eta_i$ by arc length with $d_\Sigma(\eta_1(t),\eta_2(t))<\epsilon_2/4$ for all $t\in[0,L_\eta]$. Define a map $G \colon [0,L_\eta]\to\partial N\times\partial N$ given by
\[
G(t)=(\eta_1(t),\eta_2(t)).
\]
The volume of $\partial N\times\partial N$ is $4\pi(\chi(\partial N))^2$. On the other hand, for any $t$, $G(t)$ has an embedded neighborhood $B_{\epsilon_3}(\eta_1(t))\times B_{\epsilon_3}(\eta_2(t))$ in $\partial N\times\partial N$ because $g$ is 1--Lipschitz on $\partial N$.
Such a neighborhood has volume $4\pi^2(\cosh(\epsilon_3)-1)^2$. Therefore, in search for a contradiction, if we have
\begin{equation}\label{eq:length}
    L_\eta>L_0=\frac{4\pi(\chi(\partial N))^2}{4\pi^2(\cosh(\epsilon_3)-1)^2}+2,
\end{equation}
then there exists $t_1,t_2\in [0,L_\eta]$ with $|t_1-t_2|\geq 1$ 
such that the neighborhoods described above of $G(t_1)$ and $G(t_2)$ given by products of $\epsilon_3$-disks have non-trivial intersection. In other words, we have points $\eta_1(t_1),\eta_1(t_2)$ in $\eta_1$, and points $\eta_2(t_1),\eta_2(t_2)$ in $\eta_2$ such that
\begin{itemize}
\item $d_{\eta_i}(\eta_i(t_1),\eta_i(t_2))\geq 1$ for $i=1,2$;
\item $d_{\partial N}(\eta_i(t_1),\eta_i(t_2))\leq\epsilon_2/4$ for $i=1,2$;
\item $d_\Sigma(\eta_1(t_j),\eta_2(t_j))\leq\epsilon_2/4$ for $j=1,2$.
\end{itemize}

\begin{figure}[h]
    \centering
    \includegraphics[width=.7\textwidth]{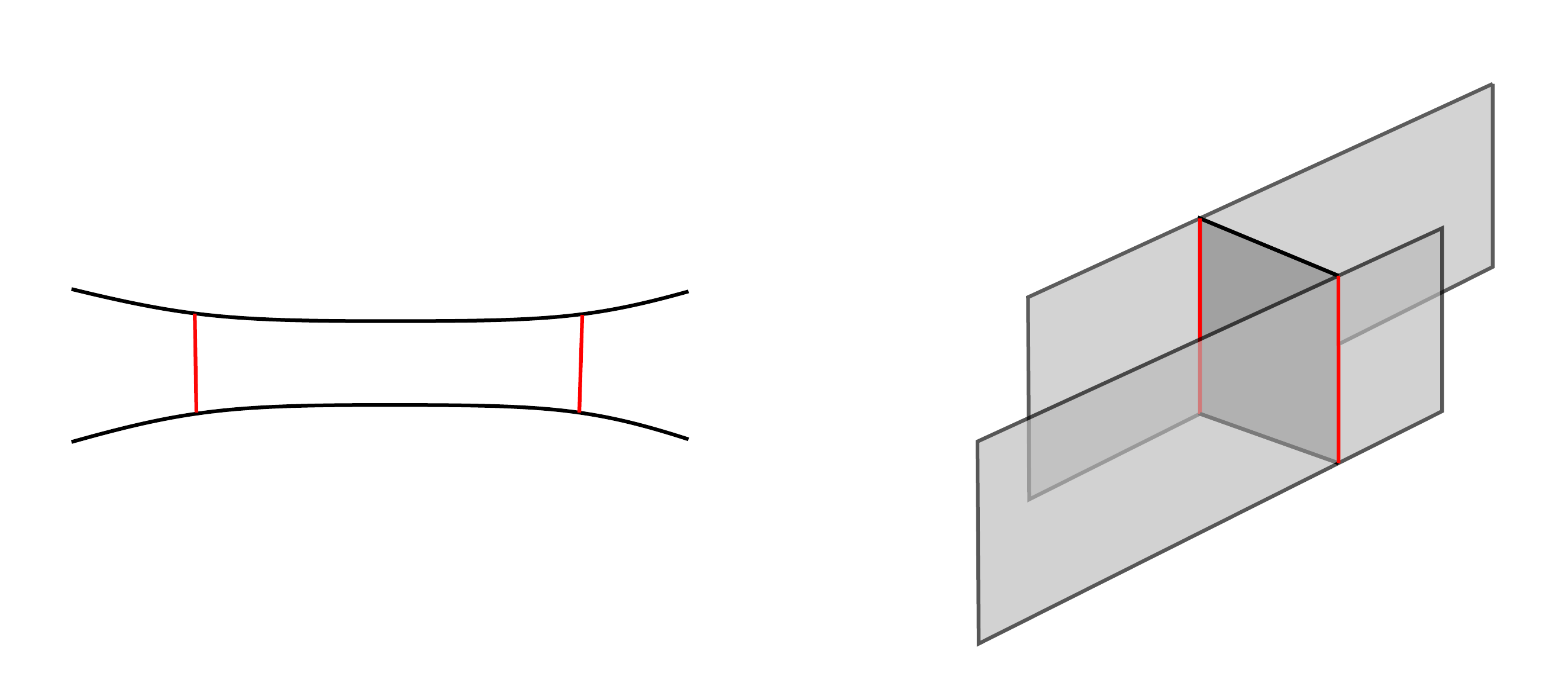}
    \caption{The band $B$ and the disk $\Delta$.}\label{fig:band}
\end{figure}

Take the shortest geodesic arcs $\gamma_1$ in $\Sigma$ connecting $\eta_1(t_1)$ and $\eta_2(t_1)$, $\gamma_2$ in $\Sigma$ connecting $\eta_1(t_2)$ and $\eta_2(t_2)$, $\gamma_3$ in $\partial N$ connecting $\eta_1(t_1)$ and $\eta_1(t_2)$, and $\gamma_4$ in $\partial N$ connecting $\eta_2(t_1)$ and $\eta_2(t_2)$. Note in particular that $\gamma_1$ is an essential arc in $\Sigma$, because $\eta_1$ and $\eta_2$ are two nearly parallel geodesic arcs. Each of the $\gamma_i$ has length $<\epsilon_2/4$, so if we let the closed curve $\gamma$ be the concatenation of $\gamma_1$, $\gamma_3$, $\gamma_2$, and $\gamma_4$, then $g(\gamma)$ is a closed curve of length $<\epsilon_2$ in $N$. But since $g(\eta_1(t_1))\in g(\eta_1)$ is in the $\epsilon_2$-thick part of $N$, $g(\gamma)$ is trivial in $N$ bounding an immersed disk $\Delta$. See \Cref{fig:band}. If we denote the sub-band of $B$ between $\gamma_1$ and $\gamma_2$ by $B'$, then $g(B')$ and $\Delta$ together gives a non-trivial immersed annulus $A$ in $N$ with $\partial A\subseteq g(\partial N)$, and one boundary of $A$ is contained in $g(X)$.

Using the homotopy from $g|_{\partial N}$ to the inclusion, we may homotope the annulus $A$ back to a properly immersed annulus $A'$ in $N$ with at least one boundary component in $X$. The annulus $A'$ has non-trivial core, because $\partial A\cap g(X)$ is the image of a loop in $X$ consisting of two distinct geodesic arcs. We claim that $A'$ is also not boundary parallel. To see this, we may assume both boundaries of $A'$ lie on $X$, otherwise there is nothing to show. If $A'$ can be homotoped to $X$, then $A$ can be homotoped to $g(X)$, which in particular means $g(\gamma_1)$ can be homotoped to $g(X)$. 
But since $\gamma_1$
is an essential arc in $\Sigma$, and since $g$ is homotopic to the identity on $\Sigma\cup X$, this contradicts the assumption that $\Sigma$ is boundary incompressible (as in the last paragraph of the proof of \Cref{claim:thin-part}). 
This contradiction shows that $A'$ is an essential annulus in $N$ with one boundary component contained in $X$. But this contradicts the assumption that $X$ is opaque. The contradiction shows that no component of $\partial\Sigma$ has length  greater than $D(L_0)$, where $L_0$ is the constant from (\cref{eq:length}) and $D(L_0)$ ultimately depends only on the topology of $\partial N$ and $\Sigma$.
\end{proof}

\section{Endgame: proof of main theorems}\label{sec:endgame}

Recall that we have fixed a closed hyperbolic $3$--manifold $M$ with an almost pseudo-Anosov flow $\varphi$, and that $S$ is an embedded surface that is transverse to $\varphi$. 

Combining what we have done to this point, we can establish the following:

\begin{proposition}\label{prop:bounded_length}
There are constants $L, d \ge 1$, depending on $|\chi(S)|$ and $\mf c(M \cut S)$, and curves $\alpha$ and $\beta$ in $\C(S)$ such that: 
\begin{itemize}
\item the lengths of the geodesic representatives of $\alpha$ and $\beta$ in $M$ are no more than $L$, and 
\item the intersection numbers $i(c^s, \alpha)$ and $i(c^u, \beta)$ are no more than $d$. 
\end{itemize}
\end{proposition}

\begin{proof}
    We will prove the proposition for $c^s$. The statement for $c^u$ is similar.

First, recall from \Cref{sec:jsj} that $\partial_w N$ is the window subsurface of $N = M \cut S$.
If $\partial^w N \cap\partial_-N$ is non-empty (as in the case where $\mf c(M\cut S) = 0$), take any component $P$ of the intersection. We will also view $P$ as a subsurface of $S$. By the construction of the windows, for any 
component $\alpha$ of $\partial P$, there exists a properly embedded essential annulus $A$ of $\mc A$ in $N$ with $\alpha \subseteq \partial A$. 
By \Cref{cor:window-distance} and \Cref{rmk:jsj} when $A$ is a product annulus, and \Cref{lem:unbalanced-annuli} otherwise, the multicurve $c^s$ must be disjoint from $\alpha$, up to isotopy.
Hence $i(c^s,\alpha)=0$, and since $\alpha$ is in $\WF(N)$, Thurston's only windows break theorem (\Cref{th:OWB}) implies that $\ell_M(\alpha)\leq C_1$, where $C_1$ depends only on $\vert \chi(S)\vert$.

If $\partial_-N$ is opaque, let $\Sigma\subset N$ to be a properly embedded essential surface transverse to (a dynamic blowup of) $\phi$ with minimal complexity. Note that $\chi(\Sigma)=-\mf{c}(M \cut S)$ by definition. Let $\alpha$ be a component of $\partial\Sigma\cap\partial_-N$. By \Cref{prop:intersection}, the intersection number $i(\alpha,c_p^s)$ is bounded by a constant $d$ depending only on $|\chi(\Sigma)|$. Moreover, since $\partial_-N$ is opaque, $c_p^s = c^s$ as any nonprincipal stable curve cobounds an essential product flow annulus (see \Cref{sec:principal}). 
By \Cref{prop:length-bound}, $\ell_M(\alpha)\leq C_2$, where $C_2$ depends on $|\chi(\Sigma)|$ and $|\chi(S)|$.
Setting $L = \max\{C_1,C_2\}$ completes the proof.
\end{proof}

Given the geometric control provided by \Cref{prop:bounded_length}, our strategy for the proofs of the main theorems is to first pass to the quasi-Fuchsian cover corresponding to $S$ and then to apply previously obtained results in that setting. 

For this, let $p \colon Q \to M$ be the cover of $M$ associated to $S$ and fix $S \to Q$ to be the unique lift of the embedding $S \to M$. Since $S$ is not a cross section of $\varphi$, the hyperbolic manifold $Q$ is quasi-Fuchsian (see \Cref{sec:transvese_shadow}); that is, it is a convex cocompact hyperbolic $3$-manifold for which $S \to Q$ is a homotopy equivalence. 

\smallskip
To apply our strategy, we will need the following topological lemma. First, recall that a $S$ is a semifiber of a manifold $M$ if $M$ is obtained by gluing together two twisted interval bundles over nonorientable surfaces and $S$ is the image of the boundary of each bundle. Such a surface becomes a fiber in a double cover; hence, a quasi-Fuchsian surface is neither a fiber nor semifiber.

\begin{lemma} \label{lem:bounded_hom}
Let $S$ be a closed, oriented incompressible surface in a closed hyperbolic $3$-manifold $M$. Suppose that $S$ is not a fiber or semifiber.

The maximal number of pairwise nonhomotopic closed curves in $S$ that become homotopic in $M$ is bounded by a constant depending only on $\vert \chi(S)\vert$.
\end{lemma}

In particular, there is a uniform bound on the number of closed geodesics in $Q$ that map to the same closed geodesic under the covering map $p \colon Q \to M$.

\begin{proof}
Fix a homotopy class $\alpha$ of curves in $S$ and let $G$ be the graph whose vertices are homotopy classes of curves in $S$ that are homotopic to $\alpha$ within $M$. Connect two distinct vertices by an edge if there is a homotopy $H \colon S^1 \times I \to M$ between these curves such that $H^{-1}(S) = S^1 \times \{0,1\}$ (call such a homotopy \emph{simple}). Note that a simple homotopy induces a homotopy $S^1\times I \to M \cut S$ and hence can be properly homotoped to be disjoint from the JSJ annuli $\mc A \subset M \cut S$ (see \Cref{sec:jsj}). 

Our goal is to bound the number of vertices of $G$. A standard argument shows that $G$ is connected. More importantly,  a theorem of Deblois \cite{deblois2016explicit} states the any reduced homotopy is a concatenation of at most $14g-12$ simple homotopies, where $g$ is the genus of $S$. In terms of $G$, this means that any two vertices of $G$ can be joined by a path of length at most $14g-12$. Hence, it suffices to bound the degree of any vertex of $G$. 

For this, consider a curve $\beta$ in the boundary of $M \cut S$. Using the JSJ theory summarized in \Cref{sec:jsj}, it is clear that $\beta$ is homotopic to at most one other curve in the boundary of $M\cut S$, up to homotopy in $\partial (M\cut S)$, unless $\beta$ is homotopic into a pared solid torus complementary component of $\mc A$. But we claim that the number of pared solid tori pieces of $(M \cut S)\cut \mc A$ is bounded by $\vert \chi(S) \vert$. This follows from the bound on the number of isotopy classes of disjoint curves on $S$ and the fact that if an annulus $A$ in $\mc A$ cobounds two pared solid tori components $X_1$ and $X_2$ of $(M \cut S) \cut \mc A$, then $X_1 \cup_A X_2$ is a (non solid torus) Seifert manifold, contradicting that $M$ is atoroidal.
This completes the proof.
\end{proof}

\begin{convention}
Given $A,B \ge 0$, we will simplify notation in what follows by writing $A \prec B$ to mean that there is a constant $c>1$, depending only on $\chi(S)$, so that $A \le c \cdot B + c$. When there is a coarse inequality that depends on other factors, like the core complexity $\mf c(M\cut S)$, we will make this dependence explicit in the notation.
\end{convention}

\subsection{Big volume}

Again let $p \colon Q \to M$ be the quasi-Fuchsian cover corresponding to $S$ and fix a lift $S \to Q$ of the inclusion $S \to M$. Applying \Cref{prop:bounded_length}, to obtain curves $\alpha$ and $\beta$ in $\C(S)$ and let $\alpha^*$ and $\beta^*$ be their geodesic representatives in $Q$ (these geodesics have the same lengths as their images in $M$). 

Consider a pleated surface $f \colon S \to Q$ homotopic to the inclusion realizing $\alpha$ (see \Cref{sec:pleated}). Recall that there is a constant $B\ge 1$, depending only on $\vert \chi(S) \vert$, called the \emph{Bers constant}, so that any hyperbolic structure on $S$ has a pants decomposition of length less than $B$ (see \cite{buser2010geometry}). Hence, we can let $P_s$ be a pants decomposition of $S$ whose length is less than $B$ with respect to the pullback hyperbolic metric from $f \colon S \to Q$. Then by the Collar Lemma (\cite{keen1974collars})
the intersection number $i(\alpha, P_s)$ is bounded by a constant $d_1\ge 0$ that depends only on $\mf c(M\cut S)$ and $|\chi(S)|$, since these are curves of bounded length on $S$. 

We can similarly finds a pants decomposition $P_u$ whose intersection number with $\beta$ is also bounded by the constant $d_1$.

Finally, we recall Hempel's bound \cite{He} (generalize by Schleimer \cite{schleimer2011notes} to the case with boundary) that for any essential multicurves $x,y$ on a surface $Y$, $d_{\C(Y)}(x,y) \le 2 \log_2(i(x,y)) +2$. Finally, we set
\[
D_{M\cut S} = \max\{d,d_1, 10\},
\]
which, from the preceding discussion, depends only on $|\chi(S)|$ and $\mf c(M \cut S)$. (Here, $d$ is the constant from \Cref{prop:bounded_length}.) The constant $10$ is to guarantee that $d_{\C(Y)}(\gamma,\delta) \le D_{M\cut S}$, whenever $\gamma$ and $\delta$ are essential curves on $S$, $Y$ is a subsurface of $S$, and $i(\gamma,\delta) \le D_{M\cut S}$.

Combining the work above with the main results of Brock \cite{Brock1}, we prove:

\begin{theorem}[Volume]\label{th:vol}
Let $M$ be a closed hyperbolic $3$-manifold, $\varphi$ an almost pseudo-Anosov flow on $M$, and $S$ a closed surface that is transverse to $\varphi$ but not a fiber.
Then
\[
d_{\C(S)}(c^s, c^u) \prec \vol(M) + D_{M\cut S}
\]
\end{theorem}

\begin{proof}
By the discussion preceding the statement of the theorem, it suffices to prove that 
\[
d_{\C(S)}(P_s,P_u) \prec \vol(M).
\]

For this, first let $\mc P(S)$ denote the pants graph of $S$. This is the graph whose vertices are isotopy classes of pants decompositions of $S$, where two pants decompositions are joined by an edge if they differ by an elementary move (see \cite{Brock1} for details). For us, it suffices to know that the coarse map $\pi_S \colon \mc P(S) \to \mc C(S)$ that associates each pants decomposition to its underlying set of curves is $2$--Lipschitz. 

Also, given a hyperbolic manifold $N$, let $\mc G_{\le L}(N)$ be the set of closed geodesics in $N$ that have length at most $L$. As before, let $B$ be the Bers constant for $S$. Brock shows in \cite[Lemma 4.2]{Brock1} that 
\[
d_{\mc P(S)}(P_s,P_u) \prec  \# \mc G_{\le B}(Q),
\]
where, as before, $p \colon Q\to M$ is the associated quasi-Fuchsian cover.
Note that the statement of \cite[Lemma 4.2]{Brock1} is for pants decompositions that are short on the hyperbolic structures `at infinity' of $Q$, but his initial reduction is to Bers short pants decompositions obtained from $1$--Lipschitz hyperbolic surfaces (see his Lemma 4.7). Alternatively, one could apply \cite[Lemma 4.3]{Brock1} after repeating his interpolation argument used in the proof of Lemma 4.2, but the argument would be exactly the same. 

We next apply \Cref{lem:bounded_hom}, together with the relation to curve graph distance, to obtained
\[
d_{\mc C(S)}(P_s,P_u) \prec  \# \mc G_{\le B}(M).
\]
The proof is now completed with a standard packing argument. For example,  \cite[Lemma 4.8]{Brock1} establishes that there exists a constant $C\ge 1$, independent of $M$, so that 
\[
\# \mc G_{\le B}(M) \le C \cdot \vol(M).
\]
Combining the inequalities completes the proof.
\end{proof}

\subsection{Big circumference}
The goal of this section is to prove:

\begin{theorem}[Circumference]
\label{th:circum}
Let $M$ be a closed hyperbolic $3$-manifold, $\varphi$ an almost pseudo-Anosov flow on $M$, and $S$ a closed surface that is transverse to $\varphi$ but not a fiber. Then there is a constant $D \ge 0$, depending only on $|\chi(S)|$ and $\mf c(M\cut S)$, so that 
\[
d_{\C(S)}(c^s, c^u) \prec  \ell(\gamma) + D,
\]
 where $\gamma$ is any closed geodesic that intersects $S$ essentially.
\end{theorem}

The strategy of the proof is similar to the volume argument, where Brock's result is replaced by 
\cite[Proposition 4.4]{aougab2022covers}
 of Aougab--Patel--Taylor. 
 Here, however, there is an additional subtlety since the lift of $\gamma$ to the quasi-Fuchsian cover $Q$ is infinite, and so we must adapt the argument to work in $M$. Throughout the argument, we will refer the reader to \cite{aougab2022covers} for details, although many noneffective versions of these facts were known earlier.

\begin{proof}
For the surface $S$, let $L_S = 2 \log(|\chi(S)|+2) \ge 0$ be the constant (depending only on $|\chi(S)|$) such that for any hyperbolic structure on $S$ and any point $p\in S$, there is an essential simple loop through $p$ whose length is at most $L_S$. See \cite[Lemma 3.1]{aougab2022covers}.

Fix a closed geodesic $\gamma$ that essentially intersects $S$, i.e. every map $S \to M$ homotopic to the inclusion intersects $\gamma$. Let $G$ be the graph whose vertices are curves in $\C(S)$ having the property that $\delta \in G^0$ if there is a hyperbolic structure $X$ on $S$ and a $1$--Lipschitz  map $f \colon X \to M$, homotopic to the inclusion $S \to M$, so that $\delta$ can be realized by a loop of length at most $L_S$ in $X$ whose image meets $\gamma \subset M$. Two curves $\delta_1$ and $\delta_2$ are joined by an edge of $G$ if they can be realized (with these properties) using the same map $f \colon X \to M$. Since, in this case, $\delta_1$ and $\delta_2$ can be realized by loops of length at most $L_S$ in some fixed hyperbolic structure $X$, we record the fact that $d_{\C(S)}(\delta_1,\delta_2) \le 20$ (\cite[Lemma 3.5]{aougab2022covers}).

Note that for any subarc $p \subset \gamma$ of length at most one, the number of distinct homotopy classes of loops (that are homotopic to simple closed curve on $S$) meeting $p$ of length at most $L_S$ is bounded by a constant $C>0$ that does not depend on $M$. (For this, see Lemma 3.6 in \cite{aougab2022covers}; if $p$ is deep enough into the thin part, where the lemma does not apply, any such loop of length less than $L_S$ must be homotopic to the core of associated Margulis tube. This is where the `homotopic to a simple closed curve on $S$' is used.) 
Moreover, using \Cref{lem:bounded_hom}, the number of distinct vertices of $G$ that correspond to loops that can be realized in $S$ with images meeting $p$ is bounded by a constant depending only on $|\chi(S)|$. 
Hence, 
\[
\#G^0 \prec \ell_M(\gamma).
\]

Next, let 
\begin{align}\label{sweep}
(f_t \colon X_t \to Q)_ {t \in [0,1]}
\end{align}
be a $1$--Lipschitz sweepout in the homotopy classes of the lift $S \to Q$, such that 
$\alpha^* \subset f_0(X_0)$ and $\beta^* \subset f_1(X_1)$. Existence of such a sweepout is essentially due to Canary; see Section 2 of \cite{aougab2022covers}. For each $t \in [0,1]$, $f_t(X_t)$ meets $\wt \gamma$ (the lift of $\gamma$ to $Q$) and hence there exists a nonempty subset $\Delta_t = \{\delta_t\}\subset G^0$ of vertices that are realized by a loops in $X_t$ of length at most $L_S$ mapping under $f_t$ to loops through points of $\wt \gamma$. Taking the induced subgraphs, the $\Delta_t$ are complete subgraphs of $G$ that determine a path from $\Delta_0$ to $\Delta_1$ in $G$; in particular, $\Delta_0$ and $\Delta_1$ are contained in the same component of $G$. (Here, we are using that the set of $t\in [0,1]$ for which $\delta \in \Delta_t$ is closed.)
Hence there is a path from $\Delta_0$ to $\Delta_1$ of length at most $\#G^0$. Since edges in $G$ correspond to curves whose distance in $\C(S)$ is at most $20$, we conclude that
\[
d_{\C(S)}(\Delta_0, \Delta_1) \prec \#G^0 \prec \ell_M(\gamma).
\] 

It now suffices to show that $d_{\C(S)}(c^s, \Delta_0) \prec D$ and $d_{\C(S)}(c^u, \Delta_1) \prec D$, for some constant $D$ depending on $\vert \chi(S) \vert$ and $\mf c(M\cut S)$. Moreover, in light of \Cref{prop:bounded_length}, we may replace $c^u$ by $\alpha$ and $c^s$ by $\beta$. 

For this, we need to know that the construction of the sweepout (\ref{sweep}) starts from a simplicial hyperbolic surface $f'_i \colon X'_i \to Q$, where $i=0,1$, such that $\ell_{X'_0}(\alpha)\le L$ and $\ell_{X'_1}(\beta)\le L$, where $L$ is from \Cref{prop:bounded_length}. Then $X_i$ is the hyperbolic metric that uniformizes the singular hyperbolic metric $X'_i$. Let $x_i$ be a loop in $X_i$ whose length is at most $L_S$ for which $f_0(x_i)$ passes through $\gamma$, i.e. $x_i \in \Delta_i$. Then since $X_i \to X'$ is $1$--Lipschitz (by \cite{ahlfors1938extension}), the curve $x_i$ has length at most $L_S$ in $X'$. By \cite[Lemma 3.3]{brock2000continuity}, the intersection number between $x_0$ and $\alpha$ (as well as the intersction number between $x_1$ and $\beta$) is bounded by a constant depending only on $|\chi(S)|$ and $L$. 
Converting from intersection numbers to curve graph distances then completes the proof.
\end{proof}

\begin{remark}[Electric length]
By further adapting arguments from \cite{aougab2022covers}, it is not hard to generalized \Cref{th:circum} to show that under the same hypotheses, 
\[
d_{\C(S)}(c^s, c^u) \prec  \ell^\epsilon_M(\gamma) + D,
\]
 where $\ell^\epsilon_M(\gamma)$ is the \emph{$\epsilon$--electric length} in $M$, which is essentially the length of the portion of $\gamma$ that stays outside of $\epsilon$--thin part of $M$. Since our focus here is recovering geometric information about $M$ from \Cref{prop:bounded_length}, we omit the details.
\end{remark}

\subsection{Short curves}
Finally, we give a short proof of the following:
\begin{theorem}[Short curves] \label{th:short_curves}
Let $M$ be a closed hyperbolic $3$-manifold with an almost pseudo-Anosov flow $\varphi$. Let $S$ be a closed surface in $M$ that is transverse to $\varphi$ and let $Y \subset S$ be a 
subsurface of $S$. Then for any $\epsilon >0$, there is a $K = K(\epsilon, \chi(S)) \ge 0$, depending only on $\epsilon$ and $\chi(S)$, such that if
\begin{itemize}
\item $d_{\C(S)}(c^{s/u} ,\partial Y) \ge 2D_{M\cut S}+2$ and 
\item $d_{\C(Y)}(c^s,c^u) \ge K + 4 \cdot  D_{M \cut S}$,
\end{itemize}
then 
\[
\ell_M(\partial Y) \le \epsilon.
\]
\end{theorem}

\begin{proof}
As before, let $Q$ be the quasi-Fuchsian cover of $M$ associated to $S$.

We begin by recalling that Minsky shows \cite[Theorem B]{minsky2001bounded} that for any $\epsilon >0$, there is a constant $K\ge0$ (depending only on $\epsilon$ and $|\chi(S)|$) such that if there are curves $\gamma,\delta$ in $\C(S)$ which essentially intersect $Y$ so that $d_{\mc C(Y)}(\gamma,\delta)\ge K$ and have length at most $B$ in $Q$, then $\ell_Q(\partial Y)\le \epsilon$. Of course, in this case, we also have $\ell_M(\partial Y)\le \epsilon$ since $Q \to M$ is $1$--Lipschitz.

We can apply this to curves $\gamma \subset P_s$ and $\delta \subset P_u$. Indeed, since $d_{\C(S)}(c^s, \gamma) \le 2D_{M\cut S}$ the first bullet in the theorem statement implies that $d_{\C(S)}(\gamma, \partial Y) \ge 2$ and so $\gamma$ essentially intersects $Y$. Similarly, the same argument also gives that $\alpha$, $\beta$, and $\delta$ also essentially intersect $Y$ since they are each within $2D_{M\cut S}$ from either $c^s$ or $c^u$ in $\C(S)$.

Moreover,
\begin{align*}
d_{\mc C(Y)}(c^s,\gamma) \le d_{\mc C(Y)}(c^s,\alpha) + d_{\mc C(Y)}(\alpha,\gamma) \le 2\cdot D_{M \cut S},
\end{align*}
as in the paragraph following the definition of $D_{M\cut S}$.
Similarly, $d_{\mc C(Y)}(c^u,\delta) \le 2 D_{M \cut S}$.

Finally, the triangle inequality gives
\begin{align*}
d_{\C(Y)}(\delta, \gamma) &\ge d_{\mc C(Y)}(c^s,c^u) - d_{\mc C(Y)}(c^s,\gamma) - d_{\mc C(Y)}(c^u, \delta)\\
&\ge (K + 4 \cdot  D_{M \cut S}) -  2 \cdot D_{M \cut S} - 2 \cdot D_{M \cut S}\\
& \ge K,
\end{align*}
and so Minsky's theorem directly applies. This completes the proof.
\end{proof}

\section{Examples and applications}
\label{sec:examples}

In this final section, we first give an application of our main theorem to finite depth foliations of hyperbolic $3$-manifolds and then conclude with a construction of flows transverse to depth one foliations where our main dynamical quantity $d_{\C(S)}(c^s,c^u)$ can be precisely controlled.

\subsection{Hyperbolic 3-manifolds with finite depth foliations}

The arguments in \Cref{sec:geom} and \Cref{sec:endgame} allow us to gain certain control of the geometry of hyperbolic 3-manifolds with a finite depth foliation, which might be of independent interest. For background on these foliations, we refer the reader to \cite{candel2000foliations}.
\smallskip

For the setup, let $M$ be a closed hyperbolic 3-manifold with a (cooriented) finite depth foliation $\FF$, which we assume to have depth at least one. This is because a depth zero foliation is a fibration over $S^1$, which is discussed in \Cref{sec:intro}.
For expository purposes, we make the simplifying assumption that the closed leaves $\FF^0$ of $\FF$ consist of a single closed connected surface $S$. In the general case, one needs to work within the individual components of $M \cut \FF^0$.

As before, we identify each component of $\partial_\pm (M \cut S)$ with $S$. In $M$, the ends of each depth one leaf $L$ spiral onto $S$ just as in the spinning construction from \Cref{sec:annuli}. So in $M\cut S$, each end of $L$ spirals onto the corresponding component of $\partial_\pm (M \cut S)$ in the same sense. This determines positive and negative ends of $L$.
In fact, the spinning process can be (non-uniquely) reversed by pushing neighborhoods of ends of $L$ into $\partial_\pm (M\cut S)$ (within collar neighborhoods of boundary components) thereby producing a surface $\Sigma$ properly embedded in $M \cut S$. This process is called de-spinning in \cite[Section 3]{landry2023endperiodic}.
Spinning $\Sigma$ back along $\partial_\pm (M\cut S)$ (as in \Cref{sec:annuli}) reproduces the leaf $L$ together with a subsurface $\Sigma_L \subset L$ associated to $\Sigma$. This subsurface is called a \emph{core} of $L$ and it determines the properly embedded surface $\Sigma$, up to proper isotopy, by isotoping its boundary back into $\partial_\pm (M\cut S)$ within a collar neighborhood. The process of going from a core subsurface $\Sigma_L \subset L$ to the properly embedded surface $\Sigma$ is also carefully described in
\cite[Section 3.2]{field2025lower}.  

In general, a core subsurface $C \subset L$ is \emph{minimal} if it 
minimizes $\vert\chi(C)\vert$ over all core subsurfaces. 
Set $\mathfrak{c}_{\FF} = \min\{\vert \chi(C) \vert\}$, where the minimum is over all cores $C$ of depth one leaves $L$ of $\mc F$.

\smallskip

Given a core $C$ of $L$ determining a properly embedded surface $\Sigma$ as above, there are naturally associated \emph{positive and negative junctures} $j^\pm \subset S$. These are the boundary components $j^\pm = \partial \Sigma \cap \partial_\pm (M\cut S)$ of the properly embedded surface $\Sigma$ considered as multicurves of $S$. Alternatively, $j^+$ is the multicurve of $S$ obtained by taking the components $\partial_+C$ of $\partial C \subset L$ that bound neighborhoods of the positive ends of $L$ and isotoping them upward into $\partial_+ (M\cut S)$ within a collar neighborhood of $\partial_+ (M\cut S)$. The negative juncture $j^-$ can be defined similarly.
When the core $C$ is minimal, we say that $j^\pm \subset S$ are \emph{minimal junctures}.

The following corollary expresses a connection between the topology of $\mc F$ and the geometry of $M$ without an explicit reference to a transverse flow.

\begin{theorem}[Junctures and geometry]\label{cor:finitedepth}
Let $M$ be a closed hyperbolic $3$-manifold with a finite depth foliation $\mc F$ so that $\mc F^0$ is a connected surface $S$. Suppose that $j^\pm  \subset S$ are minimal juncture on $S$.

There is a constant $c_S = c_S(\chi(S)) \ge 1$, depending only on $\chi(S)$, and a constant $D\ge 1$, depending only on $\chi(S)$ and $\mathfrak{c}_{\FF}$, so that the following inequalities hold:
\begin{enumerate}
\item $d_{\C(S)}(j^-, j^+) \le c_S \cdot \vol(M) + D$, and 
\smallskip
\item $d_{\C(S)}(j^-, j^+) \le c_S \cdot \ell_M(\gamma) + D$, where $\gamma$ is any closed geodesic that intersects $S$ essentially and $\ell_M(\gamma)$ is its length in $M$.
\end{enumerate}
\end{theorem}

To connect the corollary with the setup of our main theorems, we recall that by seminal (yet mostly unwritten) work of Gabai--Mosher, every finite depth foliation $\mc F$ of $M$ admits an almost pseudo-Anosov flow $\phi$ that is transverse to $\FF$. This result is currently being completed and expanded by Landry--Tsang \cite{LandryTsangStep1, LandryTsangStep2}.

\begin{proof}
Let $L$ be a depth one leaf of $\mc F$ and suppose that $C \subset L$ is a core of minimal complexity, i.e. $\vert \chi(C) \vert = \mf c_{\mc F}$. Moreover, we suppose that the minimal junctures $j^\pm$ from the statement of the corollary come from $C$. That is, if $\Sigma$ is the properly embedded surface in $N = M\cut S$ determined by $C$, then $j^\pm = \partial \Sigma \cap \partial_\pm N$. 

Now suppose that $\varphi$ is an almost pseudo-Anosov flow that is transverse to $\mc F$. The main point of the proof is to show that $\Sigma$ can be made transverse to $\varphi$, or more precisely, $C$ can be chosen in its isotopy class in $L$ so that the isotopy of $C$ moving $\partial_\pm C$ into $\partial_\pm N$ can be done transverse to $\varphi$. For this, first note that there exists a core $C' \subset L$ containing $C$ for which this is the case. This is because we can choose $C'$ so that $\partial_\pm C' \subset \mc U^\pm$. Here, $\mc U^\pm$ are the positive/negative escaping regions of $L$ as in \Cref{sec:int_bounds} consisting of points whose forward/backward orbits meet $\partial_\pm N$, respectively.

Now let $C_m$ be the smallest complexity core such that $C_m$ contains $C$, up to isotopy in $L$, and $\partial_\pm C_m \subset \mc U^\pm$.  Also, let $\Sigma_m \subset N$ be the properly embedded surface associated to the core $C_m$. Note that $\Sigma_m$ has the desired proper that it is transverse to $\varphi$. 

Suppose that $C \subset \mathrm{int}(C_m)$ but that $C$ and $C_m$ are not isotopic in $L$. Then there is a proper essential arc $\beta$ of $C_m$ contained in the complement of $C$. The homotopy that pushes $L \ssm C$ into $\partial_\pm N$ (within a collar neighborhood) shows that $\beta$ determines a boundary compressing disk for $\Sigma_m$. But by \Cref{lem:boundary-compressions}, after an isotopy of $C_m$ in $L$ (see \Cref{fig:junc}), the boundary compressing disk is isotopic to a product flow disk $D$ for $\Sigma_m$. As in \cite[Lemma 3.2]{field2025lower}, the product flow disk $D$ can be used to push a portion of $\Sigma_m$ along the flow into $\partial_\pm N$ resulting in a core $C'$ still containing $C$, up to isotopy, but with smaller complexity than $C_m$. This contradiction shows that we may assume that $C = C_m$, so that $\Sigma = \Sigma_m$ is transverse to $\varphi$. Moreover, the argument also implies that $\Sigma$ is boundary incompressible, since any boundary compression results in a smaller complexity core.

By \Cref{prop:intersection}, the distances $d_S(c^u, j^+)$ and $d_S(c^s,j^-)$ are bounded by a constant depending only on $\vert \chi(\Sigma)| = \mf c_{\mc F}$. Since $\mf c(M\cut S) \le \mf c_{\mc F}$, the required statement follows from \Cref{thm:vol-circumference}.
\end{proof}

Similarly, there is also a version of \Cref{thm:short-curves} in terms of the junctures.

\subsubsection*{Different transverse flows}
In the proof of \Cref{cor:finitedepth}, we recorded the fact that the curve graph distances $d_S(c^u, j^+)$ and $d_S(c^s,j^-)$ are bounded by a constant depending only on $\mf c_{\mc F}$. Here, $c^u$ and $c^s$ are the stable and unstable multicurves on $S$ for any fixed almost pseudo-Anosov flow that is transverse to $\mc F$. However, there may be many such flows that are \emph{not} orbit equivalent.
Since the minimal junctures $j^\pm$ are defined without reference to the flow, we have established the following:

\begin{corollary}
For $M$, $\mc F$, and $S$ as in \Cref{cor:finitedepth}, the set of stable multicurves on $S$ induced by all possible almost pseudo-Anosov flows transverse to $\FF$ is a bounded diameter subset of $\CC(S)$, where the bound depends only on $\mf c_{\mc F}$.
\end{corollary}

\medskip

\subsection{Unbounded distance between $c^s$ and $c^u$}
\label{sec:unbounded}

In this subsection, we use endperiodic mapping tori to construct examples where the curve complex graph distance between $c^s$ and $c^u$, as in the statement of \Cref{thm:vol-circumference}, can be arbitrarily large. 
We start with some basics on endperiodic maps and Handel-Miller theory. For a detailed discussion, we refer the readers to the definitive treatment by Cantwell--Conlon--Fenley \cite{CCF19}. 

Let $L$ be an infinite-type surface with finitely many ends, all of which are non-planar.
Let $f \colon L \to L$ an atoroidal endperiodic map. To make our discussion as simple as possible, we further assume that $L$ has exactly two $f$-orbits of ends. The mapping torus of $f$ admits a canonical compactification, which we denote by $N = N_f$, in which the positive end of $L$ spirals around the positive boundary component $\partial_+N$ and the negative end of $L$ spirals around the negative boundary component $\partial_-N$ (see \cite[Section 3]{Fen97} or \cite[Section 3]{field2023end}). Our simplifying assumption implies that $\partial_\pm N$ are two homeomorphic closed surfaces.

Let $h \colon \partial_+N \to \partial_- N$ be a homeomorphism reversing the induced (Stokes') orientations, and let $M(h)$ be the closed $3$--manifold obtained by gluing $\partial_+N$ to $\partial_-N$ via $h$:
\[
M(h) = \frac{N}{x \sim h(x): \forall x \in \partial_+ N}.
\]

By construction, there is an induced depth one foliation $\mc F$ of $M(h)$ whose only closed leaf is the image of $\partial_\pm N$, which we denote $S$, and whose depth one leaves are the fibers of the mapping torus $M(h) \ssm S = N \ssm \partial N$, which are all parallel to $L$.

We say the gluing map $h$ \emph{mismatches the windows} 
if no distinct isotopy classes of simple closed curves on $\partial_w N$ (see \Cref{sec:jsj}) are identified by $h$. If $h$ mismatches the windows, $M(h)$ is atoroidal and we equip $M(h)$ with the unique hyperbolic metric (using Thurston's hyperbolization theorem). As above, the foliation $\FF$ is transverse to some almost pseudo-Anosov flow $\varphi$, and we let $c^{u}$ and $c^{s}$ 
be the induced stable and unstable multicurves on $S$. Note that $M(h) \cut S = N$ and, as before, we use the same notation to denote the induced multicurves $c^u \subset \partial_+N$ and $c^s \subset \partial_-N$. 

\smallskip
Given any standard hyperbolic metric on $L$ (i.e. a hyperbolic metric such that $L$ does not contain an embedded half-plane), $f$ is isotopic to a \emph{Handel-Miller representative} $f_{\hm}$ that preserves a pair of canonical geodesic laminations $\Lambda^\pm_{\hm}$, the \emph{Handel-Miller laminations} \cite[Theorem 4.54]{CCF19}. These are the complements of the escaping sets $\cU^\mp$ of $f_{\hm}$, respectively \cite[Lemma 4.71]{CCF19}, where $\mc U^\pm$ is the set of points whose forward/backward iterates under $f_\hm$ escape compact sets of $L$. These laminations are nonempty when $f_\hm$ is not a translation of $L$ (i.e. when $N$ is not a product).

The positive escaping set $\cU^+$ is naturally an infinite-cyclic cover of $\partial_+N$ under the $\Z$-action generated by $f_{\hm}\vert_{\cU^+}$, and $\Lambda_{\hm}^+\cap\cU^+$ descends to a lamination $\Lambda_\partial^+$ on $\partial_+N$. This lamination contains closed leaves, which arise in the following way. By \cite[Theorem 6.5]{CCF19}, every frontier leaf of $\Lambda_{\hm}^-$ is periodic under $f_{\hm}$ with one or two periodic points, and for every periodic point $x$ on such a frontier leaf, there exists a positive escaping periodic ray $r\subset\Lambda_\hm^+\cap\cU^+$ based at $x$. The ray $r$ is also periodic under $f_\hm$ and so $r$ covers a closed leaf $c$ of the lamination $\Lambda_\partial^+$, and every closed leaf arises in this way. It follows that the suspension of $x$ under the power of $f_\hm$ fixing $r$, which is a closed curve in $N$, is homotopic in $N$ to $c \subset \partial_+ N$.
Similarly, we can define $\Lambda_\partial^-$ as the lamination obtained from $\Lambda^-\cap\cU^-$ under the covering $\cU^-\to\partial_-N$. We denote by $c^\pm_{\hm}$ this collection of closed leaves of $\Lambda_\partial^\pm \subset \partial_\pm N$.

\begin{lemma}\label{lem:HMLaminations}
 Let $M(h)$ be the closed manifold constructed above and let $\varphi$ be any almost pseudo-Anosov flow on $M(h)$ that is transverse to $\FF$. Then up to isotopy $c^u \subset \partial_+N$ contains $c^+_{\hm}$ and $c^s \subset \partial_- N$ contains $c^-_{\hm}$.
\end{lemma}

When $\varphi$ belongs to the special class of flows `without perfect fits,' the lemma follows from the connection between spun pseudo-Anosov and Handel--Miller maps studied in \cite[Section 8]{landry2023endperiodic}. The proof we give works in the more general setting considered here.

\begin{proof}
We begin with the following claim:
\begin{claim}
For each period point of $f_\hm$ in $\Lambda_\hm^+ \cap \Lambda_\hm^- \subset L$, its suspension is a closed curve that is homotopic in $N$ to a closed orbit of $\varphi$. 
\end{claim}

Assuming the claim, let $c$ be a component of $c^+_{\hm}$. By the discussion above, $c$ is homotopic in $N$ to the suspension of a periodic point in $\Lambda_\hm^+ \cap \Lambda_\hm^-$, which by the claim is homotopic (again in $N$) to a closed orbit of $\varphi$ in $N$. By lifting to the universal cover and applying \Cref{lem:identify-multicurves} with \Cref{rmk:outside}, we see that $c$ is homotopic in $\partial_+N$ to a component of $c^u$. 

\smallskip
We now turn to the proof of the claim. Let $f_\varphi \colon L \to L$ be the first return map for $\varphi$, which is also endperiodic and isotopic to $f_\hm$.  If $p \in \Lambda_\hm^+ \cap \Lambda_\hm^-$ is a periodic point of $f_\hm$, then we replaces $f_\hm$ (and $f_\varphi$) with the appropriate power so that $f_\hm$ fixes $p$ as well as the half-leaves of $\Lambda^\pm_\hm$ based at $p$. Let $\wt p$ be a point in the preimage of $p$ under the universal cover $\wt L \to L$ and let $\wt f_\hm \colon \wt L \to \wt L$ be the lift of $f_\hm$ that fixes $p$. Further, let $\wt f_\phi$ be the lift of $f_\varphi$ obtained by lifting an isotopy from $f_\hm$ to $f_\varphi$ starting at $\wt f_\hm$. 
Using basic covering space theory, it suffices to show that $\wt f_\varphi$ also fixes a point of $\wt L$; the required closed orbit of $\varphi$ is the orbit through the projection to $L \subset N$ of the fixed point of $\wt f_\varphi$.

By \cite[Corollary 5]{CaCo13}, $\wt f_\hm$ and $\wt f_\phi$ induce the same action on the circle $\partial \wt L$. Moreover, $\wt f_\hm$ acts on $\partial L$ with multi sink-source dynamics, meaning that there are at least $4$ isolated fixed points that alternative between sources and sinks, and each such sink/source is the endpoint of a leaf of $\wt \Lambda^\pm_\hm$ through a fixed point of $\wt f_\hm$. These facts are essentially due to \cite{CCF19} but are stated explicitly in \cite[Lemma 8.1]{landry2023endperiodic}. Hence, the lift $\wt f_\varphi$ also acts with multi sink-source dynamics with the same set of fixed points and so by
\cite[Lemma 4.9]{landry2023endperiodic} the sources and sinks for the boundary action are in fact sources and sinks for the action on the closed disk $\wt L \cup \partial \wt L$.
Finally, we apply a standard fixed point index argument (as in \cite[Lemma 4.14]{landry2023endperiodic}) to see that our fixed lift of $f_\phi$ also fixes a point of $\wt L$. This completes the proof.
\end{proof}

If we treat the induced map $\partial_-N \to S$ as the identity, then the images of $c^+_\hm$ and $c^-_\hm$ in the surface $S \subset M(h)$ are $h(c^+_\hm)$ and $c^-_\hm$, respectively. Translating \Cref{lem:HMLaminations} to the surface $S$, we have the $c^u$ contains $h(c^+_\hm)$ and $c^s$ contains $c^-_\hm$, up to isotopy in $S$. 
Since the multicurves $c_{\hm}^\pm$ are canonically associated to the endperiodic mapping torus, it is now easy to produce examples of manifolds with flows where the lower bounds in the \Cref{thm:vol-circumference} are controlled.
  This is summarized in the following:

\begin{theorem}
Let $f \colon L \to L$ be an atoroidal endperiodic map of a surface $L$ with two ends, let $N$ be its compactified mapping torus, and let $c^\pm_{\hm} \subset \partial_\pm N$ be the associated Handel--Miller curves. 

As above, for any $h \colon \partial_+ N \to \partial_-N$ that mismatches the windows, the glued manifold $M(h)$ is hyperbolic and inherits a depth one foliation $\mc F$. For any almost pseudo-Anosov flow on $M(h)$ transverse to $\mc F$, we have
\[
d_S(c_\hm^-, h(c_\hm^+))  \le d_S(c^u, c^s) \le   d_S(c_\hm^-, h(c_\hm^+)) +2.
\]
\end{theorem}

Since $d_S(c_\hm^-, h(c_\hm^+))$ can be made as large as desired using tools from the coarse geometry of $\Mod(S)$, this completes the construction. In essentially the same way, one can also obtain similar control over the size of subsurface projections, as required by \Cref{thm:short-curves}.

\bibliographystyle{alpha}
\bibliography{f&p.bib}

\end{document}